\documentclass[reqno]{amsart}
\usepackage{amssymb}
\usepackage[dvips]{epsfig}
\usepackage{graphicx}
\usepackage{color}
\usepackage{amsmath}
\usepackage{amssymb}
\usepackage{amsfonts}
\usepackage{amsthm}
\usepackage{bbm}
\usepackage{tikz}

\newcommand{\R}{\mathbb R}

\newcommand{\Z}{\mathbb Z}

\newcommand{\C}{\mathbb C}

\newcommand{\N}{\mathbb{N}}
\newcommand{\T}{\mathbb{T}}

\newtheorem{thm}{Theorem}[section]
\newtheorem{lem}[thm]{Lemma}

\newtheorem{rem}{\bf Remark}[section]
\theoremstyle{definition}
\newtheorem{defn}[thm]{Definition}

\numberwithin{equation}{section}

\begin{document}

\title[Gevrey quasi-periodic operator]{Spectral theory of the multi-frequency quasi-periodic operator with a Gevrey type perturbation}
\author{Yunfeng Shi}
\address[Y. Shi] {College of Mathematics,
Sichuan University,
Chengdu 610064,
China}
\email{yunfengshi@scu.edu.cn, yunfengshi18@gmail.com}
\date{\today}


\begin{abstract}
In this paper  we study the multi-frequency quasi-periodic operator with a Gevrey type perturbation. We first establish the large deviation theorem (LDT) for the multi-dimensional  operator with a sub-exponential (or Gevrey) long-range hopping, and then prove the pure point spectrum property. Based on the LDT and the Aubry duality, we show the absence of   point spectrum for the $1D$ exponential long-range  operator with a multi-frequency and a Gevrey  potential. We also prove the spectrum has positive   Lebesgue measure.
\end{abstract}

\maketitle





\tableofcontents

\maketitle
\section{Introduction and main results}
In this paper we study the spectral properties of  the multi-frequency long-range quasi-periodic operator with a Gevrey type perturbation. More precisely,  we first consider the multi-dimensional quasi-periodic operator with a  Gevrey long-range hopping and  an analytic potential which satisfies the \textit{non-degeneracy} condition. We prove such operator has pure point spectrum (with sub-exponentially decaying eigenfunctions) in  the large coupling regime (see Theorem \ref{mthm2} in the following).  The  Aubry duality of this operator is a $1D$ multi-frequency quasi-periodic operator with an exponential long-range hopping and a Gevrey potential. We show the absence of point spectrum for the Aubry duality in the small coupling regime  (see Theorem \ref{mthm1} in the following). We want to mention that in the small coupling regime the \textit{non-degeneracy} assumption on the potential is not needed. While we can prove the absence of  point spectrum for the Aubry duality, we can not obtain the existence of  absolutely continuous (ac) spectrum via the present method.  The Kotani's theory suggests that the existence of  ac spectrum  implies the positivity of the Lebesgue measure of the spectrum, which  motivates us to study the measure of the spectrum (see  Theorem \ref{cor1} in the following) in the small coupling regime.

We start with the long-range hopping, which is a Toeplitz operator. Let $h: \R^d/\Z^d=\T^d\to\R$ be a function. We define the Toeplitz operator  (with the symbol $h$) as
\begin{align*}
\mathcal{T}_h(m,n)=\widehat h_{m-n},\ m,n\in\Z^d,
\end{align*}
where $\widehat h_n=\int\limits_{\T^d}h(\theta)e^{-2\pi in\cdot\theta}d\theta$. We also define $\delta_{mn}= 1$ if $m=n$, and $\delta_{mn}=0$ if $m\neq n$.
 \subsection{Pure point spectrum}
We first study the multi-dimensional operator with a Gevrey long-range hopping and an analytic potential.

Assume that $v$ is Gevrey regular, i.e.,  $v(x)\in C^\infty(\T^d,\R)$ satisfies for some $\gamma\in(0,1]$ and $\forall\ n\in\Z^d$,
\begin{align}\label{gvc}
|\widehat v_n|\leq e^{-\rho|n|^\gamma},
\end{align}
where $\rho>0$, $|n|=\sup\limits_{1\leq i\leq d}|n_i|$. Notice that $v$ is analytic if $\gamma=1$.

We consider the operator
\begin{align}\label{lro}
\widetilde H_{\lambda f,\omega,\theta}&=\lambda \mathcal{T}_v+ f(\theta+n\omega)\delta_{nn'}, \ \theta\in\T^d,\\
\label{ds} n\omega&=(n_1\omega_1,\cdots,n_d\omega_d),
\end{align}
where  $\lambda^{-1}>0$ is the coupling, $\theta$ is the phase, $\omega\in\T^d$ is the frequency and $f$ is a real analytic function satisfying the \textit{non-degeneracy} condition: For all $j=1,\cdots,d$ and $$\theta_j^\neg=(\theta_1,\cdots,\theta_{j-1},\theta_{j+1},\cdots,\theta_d)\in\T^{d-1},$$ the map
\begin{align*}
\theta_j\mapsto f(\theta_j,\theta_j^\neg)
\end{align*}
is a {non-constant} function of $\theta_j\in\T$.

We have
\begin{thm}\label{mthm2}
Let $\widetilde H_{{\lambda} f,\omega,\theta}$ be defined by \eqref{lro}--\eqref{ds} with $v$ satisfying \eqref{gvc} and $f$ satisfying the non-degeneracy condition. Then  for any $\varepsilon>0$,  there exists a $\lambda_0=\lambda_0(d,\gamma,\rho,f,\varepsilon)>0$ such that
the following holds: For $0\leq\lambda\leq \lambda_0$ and $\theta\in\T^d$, there exists some $\Omega=\Omega(d,\gamma,\rho,\lambda f,\theta)\subset \T^d$ with $\mathrm{mes}(\Omega)\leq \varepsilon$ such that, if $\omega\in\T^d\setminus\Omega$,  then
$\widetilde H_{{\lambda} f,\omega,\theta}$ has pure point spectrum with sub-exponentially (exponentially if $\gamma=1$) decaying eigenfunctions.
\end{thm}

\begin{rem}
This theorem extends the result of Bourgain \cite{BG} to the Gevrey long-range hopping case.
\end{rem}

 The search for  nature of the spectrum and the behaviour of the eigenfunctions for the $1D$ quasi-periodic operator  has attracted great attention over years. Of particular importance is the phenomenon of the Anderson localization (AL), {where we say an operator satisfies AL if it has   pure point spectrum with exponentially decaying eigenfunctions.
The early results on the AL  were perturbative and restricted on ``cos'' type potentials  \cite{SJ,FSWC}. The first {non-perturbative}\footnote{Here, by a {non-perturbative} argument  we mean the argument allows the size of the perturbation to be independent of the frequency.} AL  was obtained by Jitomirskaya \cite{Jit94} in the almost Matheiu operator (AMO) setting. By developing a new type of KAM arguments, Eliasson \cite{EA}  proved  pure point spectrum for the $1D$  Schr\"odinger operator with a large Gevrey potential. Eliasson's result is  perturbative and needs the potential to satisfy some \textit{transversality} condition.  Later, the celebrated work of Jitomirskaya \cite{JA} indicated that   the AL can  hold for the AMO with the Diophantine condition if the coupling $\lambda>1$.
Significantly, Bourgain-Goldstein \cite{BGA} established the {non-perturbative} AL for the $1D$ Schr\"odinger operator with a single-frequency and an analytic potential. Klein \cite{KJ,KJ1} proved the AL for the $1D$  Schr\"odinger operator with a Gevrey potential. 
 In the long-range setting Bourgain-Jitomirskaya \cite{BJI} proved the {non-perturbative} AL for the exponential long-range operator with a ``cos'' potential. In \cite{BB} Bourgain extended  the result of \cite{BJI} to an operator with an analytic potential. An improvement of some long-range estimates of \cite{BJI} has recently been  established by Avila-Jitomirskaya \cite{AJ10}. We also mention the work of Jian-Shi-Yuan \cite{JSY} in which a {non-perturbative} AL was obtained for some $1D$ long-range block operator. We refer the reader to {\cite{AJA,AYZD,JLA} }for more recent AL results in the $1D$ setting.

  In the multi-dimensional case only the {perturbative} localization can  be expected \cite{Bou021}. The first  multi-dimensional localization was obtained by Chulaevsky and Dinaburg \cite{CD93} for a single-phase operator with an exponential long-range hopping. Their {perturbative} KAM  methods seem not applicable in the multi-phase case. Bourgain, Goldstein and Schlag \cite{BGSA} developed a new way to combine the multi-scale analysis developed by Fr\"ohlich-Spencer \cite{FS83} and some of the {non-perturbative} methods to the case $(n,\theta,\omega)\in\Z^2\times\T^2\times\T^2$, and obtained the AL for the large analytic potential. To perform such multi-scale analysis, the sub-linear growth  of the number of ``bad'' small boxes contained in a big box  becomes essential. In the single-phase case only Diophantine condition of the frequency can ensure the sub-linearity property.  In the $(n,\theta,\omega)\in\Z^2\times\T^2\times\T^2$ case  to get the sub-linearity property, additionally arithmetic conditions on the frequency are needed \cite{BGSA}.  It was also shown by Bourgain \cite{BA} that the Diophantine property of the frequency of the skew shift is  sufficient to guarantee the sub-linearity property.  For $(n,\theta,\omega)\in\Z^d\times\T^d\times\T^d$ with $d\geq 3$,   it is difficult to ensure the sub-linearity property as in the case $d\leq2$ (or $dD$ with the single-phase). To overcome this problem, Bourgain \cite{BG} introduced new methods and allowed the eliminations of the frequency to depend on the potential when proving the LDT.  This enables him to extend results of \cite{BGSA} to \textit{arbitrary} dimension $d$. The basic techniques of \cite{BG} are also  semi-algebraic sets arguments and matrix-valued Cartan's estimates, but need more delicate analysis. Recently, methods of Bourgain \cite{BG} have been largely extended  by Jitomirskaya-Liu-Shi \cite{JLS} to  the long-range  operator with $(n,\theta,\omega)\in\Z^d\times\T^b\times\T^b$  for \textit{arbitrary} $b,  d$. The result of \cite{JLS} is significantly more general and more technically complex, and can also be viewed as both a clarification and at the same time streamlining of \cite{BG}.
  We also mention the work of Bourgain-Kachkovskiy \cite{BKG} in which the case $(n,\theta,\omega)\in\Z^2\times\T^2\times\T$ was studied. For the multi-dimensional long-range operator  with a ``cos'' potential, localization results
 with the {fixed Diophantine  frequency} have been obtained  by Jitomirskaya-Kachkovskiy \cite{JKM} and Ge-You-Zhou \cite{GYZarxiv}.




\subsection{Absence of point spectrum}
We then study the Aubry duality of \eqref{lro} in the case $f(\theta+n\omega)=g(\theta+n\cdot\omega)$,  where $\theta\in\T, n\cdot\omega=\sum\limits_{i=1}^dn_i\omega_i$ and
$g$ is a non-constant real analytic function.
This leads to the  $1D$ exponential long-range quasi-periodic  operator
\begin{equation}\label{qps}
H_{\lambda v,\omega,x}=\mathcal{T}_g +\lambda v(x+\ell\omega)\delta_{\ell\ell'},\ x\in\T^d, \ell,\ell'\in\Z,
\end{equation}
where $v$ is defined by \eqref{gvc}.

If $g(\theta)=2\cos2\pi\theta$, the operator  (\ref{qps}) becomes the standard multi-frequency quasi-periodic Schr\"odinger operator.
In particular, {we} call (\ref{qps}) the almost Mathieu operator  if it is  a $1D$ quasi-periodic Schr\"odinger operator satisfying $v(x)=2\cos2\pi x$.

Denote by $\mathrm{mes}(\cdot)$ the Lebesgue measure. We have
\begin{thm}\label{mthm1}
Let $H_{\lambda v,\omega,x}$ be defined by \eqref{qps} with $v$ satisfying \eqref{gvc}. Then for any $\varepsilon>0$, there exists a $\lambda_0=\lambda_0(g,d,\gamma,\rho,\varepsilon)>0$ such that the following holds:
For $0\leq\lambda\leq \lambda_0$,  there exists some $\Omega=\Omega(g,d,\gamma,\rho,\lambda)\subset \T^d$ with $\mathrm{mes}(\Omega)\leq \varepsilon$ such that, if $\omega\in\T^d\setminus\Omega$, then $H_{\lambda v,\omega,x}$ has no  point spectrum for all $x\in\T^d$.
\end{thm}
\begin{rem}
 {The \textit{non-degeneracy} condition on $v$ is not needed here. In addition, we think the operator should have pure ac spectrum if $0<\lambda\ll1$.} 
\end{rem}

  {In the following we review some results on the ac spectrum}.   Consider first the  one-frequency operator (i.e. $d=1$) case.  As is well-known, the spectrum of the free Laplacian   on $\Z$ is pure ac.  Thus the question whether pure {ac} spectrum property holds for the Schr\"odinger operator with a small quasi-periodic potential naturally arises. Early results  were restricted on the AMO case \cite{BLTC,CDJ}. In {the} continuous setting Eliasson \cite{EC} proved pure {ac} spectrum for a Schr\"odinger operator with the Diophantine frequency and a small analytic quasi-periodic potential by using the KAM  scheme (see \cite{Amo09} for the discrete case). Later, Bourgain-Jitomirskaya \cite{BJI} developed a {non-perturbative} argument to handle the $1D$ discrete Schr\"odinger operator with a small analytic quasi-periodic potential for a.e. phase $x\in\T$.  Puig \cite{Pui06} improved partial results of Eliasson \cite{EC}  to the non-perturbative and discrete setting.   The proof of Puig was based on the Aubry duality and a non-perturbative localization result in the exponential long-range Hamiltonian in \cite{BJI}. Significantly, Avila-Jitomirskaya \cite{AJ10} developed a quantitative version of the duality based on the dual concepts of almost reducibility and almost localization,  which ultimately implied a  non-perturbative pure ac spectrum result {holds} for the $1D$ analytic Schr\"odinger operator with the  Diophantine frequency for all phase $x\in\T$. { If $0\leq \lambda<1$, Avila-Damanik \cite{AD08} proved the pure ac spectrum of the AMO for every irrational frequency and for a.e. $x\in\T.$  We also mention the work \cite{AFK11} in which the existence of ac spectrum is obtained for a $1D$ analytic Schr\"odinger operator with any irrational frequency. Remarkably, Avila \cite{Avi10,AviARC} even established the Almost Reducibility Conjecture and proved the pure ac spectrum for the analytic quasi-periodic Schr\"odinger operator in the subcritical regime.}

{Much less is known about the multi-frequency quasi-periodic Schr\"odinger operator. Based on arguments of \cite{FK09}, Bjkerl\"ov-Krikorian \cite{BKarxiv}  showed the existence of ac spectrum for a smooth multi-frequency quasi-periodic Schr\"odinger operator  without the smallness restriction on the potential.  Recently, Hou-Wang-Zhou \cite{HWZ20} proved the existence of ac spectrum for the analytic multi-frequency quasi-periodic Schr\"odinger operator with a Liouville frequency. Very recently, Cai \cite{Cai21} obtained the pure ac spectrum for the multi-frequency quasi-periodic Schr\"odinger operator with a finitely differentiable potential relying on the  almost reducibility results of  \cite{CCYZ}.}




\subsection{Lebesgue measure of the spectrum}
It is well-known that the spectrum of $H_{\lambda v,\omega,x}$ is independent of $x\in\T^d$   if $1$ and $\omega$ are rationally independent.  In this case we denote by $\Sigma_{\lambda v,\omega}$ the spectrum of $H_{\lambda v,\omega,x}$.
We have
\begin{thm}\label{cor1}
Let $v$ satisfy \eqref{gvc} and let $g$ be a non-constant analytic function. Then  for any $\varepsilon>0$,  there exists a $\lambda_0=\lambda_0(g,d,\gamma,\rho,\varepsilon)>0$ such that
the following holds: For $0\leq\lambda\leq \lambda_0$, there exists some $\Omega=\Omega(g,d,\gamma,\rho,\lambda)\subset \T^d$ with $\mathrm{mes}(\Omega)\leq \varepsilon$ such that, if $\omega\in\T^d\setminus\Omega$,  then
\begin{align*}
\mathrm{mes}(\Sigma_{\lambda v,\omega})\geq c>0,
\end{align*}
where $c=c(\lambda_0)$.
\end{thm}

\begin{rem}
 The study of the Lebesgue measure of  the spectrum for the quasi-periodic operator has a long history. The famous Aubry-Andr\'e conjecture \cite{AAA} states that the measure of the AMO is exactly $|4-4\lambda|$ for all frequency $\omega\in \R\setminus\mathbb{Q}$. Before \cite{AKA}, only partial results were obtained \cite{HSM,AMSC,LC,JKM1}. Remarkably, Avila-Krikorian \cite{AKA} settled this conjecture completely. We would also like to mention the  recent elegant work \cite{JK19}, where  a short new  proof of zero measure of the spectrum for the critical (i.e. $\lambda=1$)  AMO was given. If one considers the more general Schr\"odinger operator, there is no explicit representation of the measure of the spectrum. However, based on the LDT and semi-algebraic sets arguments, Bourgain \cite{BB} was able to prove that the Lebesgue measure of the spectrum for the $1D$  Schr\"odinger operator with a single-frequency and an analytic potential is strictly positive. Bourgain's result is non-perturbative. In the present we extend Bourgain's result to the multi-frequency operator with a Gevrey potential and a long-range hopping (but {perturbative}).
 \end{rem}



\subsection{Perturbative essentials}
As mentioned above,  our results and methods are {perturbative}. Actually, even in the $1D$ Gevrey perturbation case, only  perturbative results could be expected.  Due to the relatively lower regularity (resp. slower decaying) of the potential (resp. long-range hopping), it seems that only perturbative methods (such as the multi-scale analysis) are applicable. In fact,  the appropriate  estimates on the Green's functions are key to establish the above spectral results. We can restrict our consideration to the case $(n,\theta,\omega)\in\Z\times\T\times \T$. We denote by $\widetilde H_N(\theta)$ the restriction of $\widetilde H_{\lambda f,\omega,\theta}$ on $[-N,N]\subset \Z$. Following the non-perturbative techniques (without any inductive arguments) of \cite{BJI,BB}, the Green's function $G_N(E;\theta)=(\widetilde H_N(\theta)-E)^{-1}$ can be represented via the Cramer's rule as
\begin{align*}
G_N(E;\theta)(m,n)=\frac{\mathcal{M}_{m,n}}{\det(\widetilde H_N(\theta)-E)},
\end{align*}
where $\mathcal{M}_{m,n}$ is the $(m,n)$-minor of $\widetilde H_N(\theta)-E$. As in \cite{BJI,BB}, one may show
\begin{align*}
|\det(\widetilde H_N(\theta)-E)|\sim e^{N\int_{\T}\log|f(\theta)-E|\mathrm{d}\theta+o(\lambda)N}
\end{align*}
for $\theta$ being outside a set of measure at most $e^{-N^c},c\in(0,1)$.
Due to the sub-exponentially decaying of $\widehat v_n$, the best possible upper bound of $\mathcal{M}_{m,n}$ may be
\begin{align*}
|\mathcal{M}_{m,n}|\leq e^{-\rho|m-n|^\gamma+N\int_{\T}\log|f(\theta)-E|\mathrm{d}\theta+o(\lambda)N}.
\end{align*}
Consequently,
\begin{align*}
|G_N(E;\theta)(m,n)|\leq  e^{-\rho|m-n|^\gamma+o(\lambda)N}.
\end{align*}
In the case of $\gamma\in (0,1)$, no off-diagonal decay of $G_N(E;\theta)$ could be expected for $0<\lambda\leq\lambda_0$. This technical difficulty is the main motivation of the present paper to use  methods developed by Bourgain \cite{BG} and Jitomirskaya-Liu-Shi \cite{JLS}, which depend mainly on the multi-scale analysis. That of course will  lead to  perturbative results.
\subsection{Strategy of the proofs}
 We outline the proofs. First, we will prove the LDT for Green's functions of $\widetilde H_{\lambda f,\omega,\theta}$. This  depends on the multi-scale analysis developed in \cite{BG,JLS}. The matrix-valued Cartan's estimates and semi-algebraic geometry arguments play essential roles in this step. 
In  \cite{JLS}  the authors considered the multi-dimensional quasi-periodic operator with the exponentially decaying long-range hopping (which deals with the more complicated $b$-frequency setting). It turns out that the Gevrey long-range hopping case needs to improve some arguments  of  \cite{JLS}:
 \begin{itemize}
 \item[$\bullet$] In the proof of the resolvent identity (see the Appendix for details) it needs the off-diagonal decaying speeds of the Green's functions to depend on the Gevrey index $\gamma$.  In the proof of the LDT it also needs to give more delicate estimates on various parameters. The key idea is to remove more $\theta$ in some sense when establishing the LDT . This depends sensitively on the Gevrey index $\gamma$ as well.
 \item[$\bullet$]Furthermore, the sub-linear growth property in our setting becomes more precise, which heavily relies  on $\gamma$.
 \end{itemize}

 To prove the pure point spectrum (i.e. Theorem \ref{mthm2}), it suffices to eliminate the energy in LDT and then apply the Shnol's Theorem. This will be finished by using semi-algebraic sets arguments (including Yomdin-Gromov triangulation Theorem) as in \cite{BG}.

  To prove the absence of point spectrum (i.e. Theorem \ref{mthm1}), we will combine the LDT  with a trick originated from Delyon \cite{Del87}.  In contrast with \cite{EA,KJ,KJ1}, our result holds without any \textit{transversality} restriction  on the Gevrey potential. The proofs of \cite{KJ,KJ1} dealt with the Schr\"odinger operator with a Gevrey potential directly. To prove the LDT, Klein performed an inductive scheme as in \cite{BGA,BGS01} and  needed the \textit{transversality} condition of the potential to guarantee the validity of the initial step (or a {\L}ojasiewicz type inequality). Instead, in the present we  establish the LDT for  the Aubry dual operator of \eqref{qps}. It turns out this  operator is actually a multi-dimensional quasi-periodic operator with an analytic potential and a Gevrey  long-range {hopping}.

 To prove the spectrum has positive measure (i.e. Theorem \ref{cor1}), we will use a renormalization scheme of Bourgain \cite{BB} relying on the complexity estimates.  In \cite{BB} Bourgain directly applied the LDT of \cite{BGA} together with semi-algebraic sets arguments (including Tarski-Seidenberg principle and bounds on the Betti numbers) to construct sufficiently many approximate eigenvalues.  However, for  the Schr\"odinger operator with a Gevrey potential, the only known LDTs were proved by Klein \cite{KJ,KJ1}, but require the potential to satisfty the \textit{transversality} condition. Moreover, Klein's methods seem invalid in the long-range case. To overcome these difficulties, we again use the powerful Aubry duality. Precisely, by {the} well-known result (see \cite{Pui06,JKM}), we have $\Sigma_{\lambda v,\omega}=\widetilde{\Sigma}$, here $\widetilde{\Sigma}$ denotes the spectrum of the Aubry duality of \eqref{qps}. It turns out this Aubry duality is a multi-dimensional Gevrey long-range operator with an analytic potential. Bourgain \cite{BB} claimed that his arguments  remain valid  for the long-range operator in the $1D$ and single-frequency case once the LDT was established. In this paper we extend Bourgain's method to the multi-dimensional  case.

\subsection{The structure of this paper}
The structure of the paper is as follows. Some preliminaries are introduced in \S2. The LDT is established in \S3. In \S4, \S5 and \S6, we finish the proof of Theorems \ref{mthm2}, \ref{mthm1} and \ref{cor1}, respectively. Some key estimates are included in the Appendix.

\section{Preliminaries}

\subsection{The notations} Let $a>0,b>0$. We define $a\lesssim b$ (resp. $a\ll b$) if there is some $\varepsilon>0$ (resp. small $\varepsilon>0$)  so that $a\leq \varepsilon b$. We write $a\sim b$ if $a\lesssim b$ and $b\lesssim a$. We write $a\pm$ to denote $a\pm\varepsilon$ for some small $\varepsilon>0$.

For any $x\in\mathbb{R}^d$, let $|x|=\max\limits_{1\leq i\leq d}|x_i|$. For $\Lambda\subset\mathbb{R}^d$,  we introduce
$$\mathrm{diam}(\Lambda)=\sup_{n,n'\in \Lambda}|n-n'|, \ \mathrm{dist}(m,\Lambda)=\inf_{n\in \Lambda}|m-n|.$$


For $\theta\in\R^d$ and $1\leq j\leq d$, let ${\theta}_j^\neg=(\theta_1,\cdots,\theta_{j-1},\theta_{j+1}\cdots,\theta_d)\in\R^{d-1}$.

For $x\in\mathbb{R}^{d_1}$ and $\emptyset\neq X\subset\mathbb{R}^{d_1+d_2}$,  define the $x$-section of $X$ to be
$$X(x)=\{y\in\mathbb{R}^{d_2}:\  (x,y)\in X\}.$$ For example, $X({\theta}_j^\neg)=\{\theta_j\in\T:\ (\theta_j,\theta_j^\neg)\in X\}$ if $\emptyset \neq X\subset\T^d$.

For $x\in\R$, we denote by $[x]$ its integer part.

Throughout this paper, we assume $\rho\in (0,1)$ for simplicity.

\subsection{Some facts on semi-algebraic sets}
\begin{defn}[Chapter 9, \cite{BB}]
A set $\mathcal{S}\subset \mathbb{R}^n$ is called a semi-algebraic set if it is a finite union of sets defined by a finite number of polynomial equalities and inequalities. More precisely, let $\{P_1,\cdots,P_s\}\subset\mathbb{R}[x_1,\cdots,x_n]$ be a family of real polynomials whose degrees are bounded by $d$. A (closed) semi-algebraic set $\mathcal{S}$ is given by an expression
\begin{equation}\label{smd}
\mathcal{S}=\bigcup\limits_{j}\bigcap\limits_{\ell\in\mathcal{L}_j}\left\{x\in\mathbb{R}^n: \ P_{\ell}(x)\varsigma_{j\ell}0\right\},
\end{equation}
where $\mathcal{L}_j\subset\{1,\cdots,s\}$ and $\varsigma_{j\ell}\in\{\geq,\leq,=\}$. Then we say that $\mathcal{S}$ has degree at most $sd$. In fact, the degree of $\mathcal{S}$ which is denoted by $\deg(\mathcal{S})$, is the  smallest $sd$ over all representations as in $(\ref{smd})$.
\end{defn}

\begin{lem}[Tarski-Seidenberg Principle, \cite{BB}]\label{tsp}
Denote by $(x,y)\in\mathbb{R}^{d_1+d_2}$ the product variable. If $\mathcal{S}\subset\mathbb{R}^{d_1+d_2}$ is semi-algebraic of degree $B$, then its projections $\mathrm{Proj}_x\mathcal{S}\subset\mathbb{R}^{d_1}$ and
 $\mathrm{Proj}_y\mathcal{S}\subset\mathbb{R}^{d_2}$ are semi-algebraic of degree at most $B^{C}$, where $C=C(d_1,d_2)>0$.
\end{lem}

\begin{lem}[\cite{BB}]\label{btb}
Let $\mathcal{S}\subset \R^d$ be a semi-algebraic set of degree $B$. Then the sum of all Betti numbers of $\mathcal{S}$ is bounded by $B^{C}$, where $C=C(d)>0$.
\end{lem}

\begin{lem}[\cite{BG}]\label{proj}
Let $\mathcal{S}\subset[0,1]^{d=d_1+d_2}$ be a semi-algebraic set of degree $\deg(\mathcal{S})=B$ and $\mathrm{mes}_d(\mathcal{S})\leq\eta$, where
\begin{equation*}
\log B\ll \log\frac{1}{\eta}.
\end{equation*}
Denote by $(x_1,x_2)\in[0,1]^{d_1}\times[0,1]^{d_2}$ the product variable. Suppose
$$ \eta^{\frac{1}{d}}\leq\varepsilon.$$
Then there is a decomposition of $\mathcal{S}$ as
$$\mathcal{S}=\mathcal{S}_1\cup\mathcal{S}_2$$
with the following properties. The projection of $\mathcal{S}_1$ on $[0,1]^{d_1}$ has small measure
$$\mathrm{mes}_{d_1}(\mathrm{Proj}_{x_1}\mathcal{S}_1)\leq B^{C(d)}\varepsilon,$$
and $\mathcal{S}_2$ has the transversality property
$$\mathrm{mes}_{d_2}(\mathcal{L}\cap \mathcal{S}_2)\leq B^{C(d)}\varepsilon^{-1}\eta^{\frac{1}{d}},$$
where $\mathcal{L}$ is any $d_2$-dimensional hyperplane in $[0,1]^d$ s.t.,
$$\max\limits_{1\leq j\leq d_1}|\mathrm{Proj}_\mathcal{L}(e_j)|<{\varepsilon},$$
where we denote by $e_1,\cdots,e_{d_1}$ the $x_1$-coordinate vectors.
\end{lem}
In \cite{BG}, Bourgain proved a  result for eliminating multi-variables.
\begin{lem}[Lemma 1.18, \cite{BG}]\label{svl}
Let $\mathcal{S}\subset [0,1]^{d+r}$ be a semi-algebraic set of degree $B$ and such that
$$\mathrm{mes}_d(\mathcal{S}(y))<\eta\ \mathrm{for}\ \forall\  y\in [0,1]^r.$$
Then the set
\begin{equation*}
\left\{(x_1,\cdots,x_{2^r})\in [0,1]^{d2^r}:\  \bigcap\limits_{1\leq i\leq 2^r}\mathcal{S}(x_i)\neq\emptyset\right\}
\end{equation*}
is semi-algebraic of degree at most $B^{C}$ and measure at most
\begin{equation*}
B^{C}\eta^{d^{-r}2^{-r(r-1)/2}},
\end{equation*}
where $C=C(d,r)>0$.
\end{lem}

\begin{lem}[Lemma 1.20, \cite{BG}]\label{rmf}
Let $\mathcal{S}\subset [0,1]^{dr}$ be a semi-algebraic set of degree $B$ and $\mathrm{mes}(\mathcal{S})<\eta$ with $\eta>0.$

For  $\omega= (\omega_1,\cdots,\omega_r)\in[0,1]^{r}$ and $ n=(n_1,\cdots,n_r)\in\mathbb{Z}^r$, define
$$n\omega=(n_1\omega_1,\cdots, n_r\omega_r).$$

For any $C>1$, define  $\mathcal{N}_1,\cdots,\mathcal{N}_{d-1}\subset \mathbb{Z}^r$ to be finite sets with the following property:
$$\min\limits_{1\leq s\leq r}|n_s|>(B\max\limits_{1\leq s\leq r}|m_s|)^C, $$
where $n\in \mathcal{N}_{i}, m\in\mathcal{N}_{i-1}\  (2\leq i\leq d-1)$.

 Then there is some $C=C(r,d)>0$ such that for
$\max\limits_{n\in\mathcal{N}_{d-1}}|n|^C<\frac{1}{\eta},$
one has
\begin{align*}
&\ \mathrm{mes}(\{\omega\in [0,1]^{r}:\  \exists \ n^{(i)}\ \in\mathcal{N}_i\  s.t.,\  (\omega,n^{(1)}\omega,\cdots,n^{(d-1)}\omega)\mod\mathbb{Z}^{dr}\in \mathcal{S}\})\\
&\ \ \ \ \ \leq B^{C}\delta,
\end{align*}
where
$$\delta^{-1}=\min\limits_{n\in\mathcal{N}_1}\min\limits_{1\leq s\leq r}|n_s|.$$
\end{lem}
\section{LDT of Green's functions}
If $\Lambda\subset \Z^d$, we denote $\widetilde{H}_{\Lambda}(\theta)=R_{\Lambda}{\widetilde{H}_{\lambda f,\omega,\theta}}R_{\Lambda}$,  where $R_{\Lambda}$ is the restriction operator. Define the Green's function as
\begin{align*}
G_{\Lambda}(E;\theta)=(\widetilde{{H}}_{\Lambda}(\theta)-E+i0)^{-1}.
\end{align*}
We denote by $Q_N$ an elementary region of size $N$ centered at $0$ (see \cite{JLS}), which is one of the following regions:
\begin{equation*}
  Q_N=[-N,N]^d
\end{equation*}
or
$$Q_N=[-N,N]^d\setminus\{n\in\mathbb{Z}^d: \ n_i\varsigma_i 0, 1\leq i\leq d\},$$
where  for $ i=1,2,\cdots,d$, $ {\varsigma_i\in \{\{n<0\},\{n>0\},\emptyset\}}$ and at least two $ \varsigma_i$  are not $\emptyset$.
Denote by $\mathcal{E}_N^{0}$ the set of all elementary regions of size $N$ centered at $0$. Let $\mathcal{E}_N$ be the set of all translates of  elementary regions, namely,
$$\mathcal{E}_N:=\bigcup\limits_{n\in\mathbb{Z}^d,Q_N\in \mathcal{E}_N^{0}}\{n+Q_N\}.$$

The main result of this section is  
\begin{thm}[LDT]\label{ldt}
Fix any  $0<c_1\ll\gamma$.  Then there exist $\underline{N}_0=\underline{N}_0(d,\rho,\gamma,f,c_1)$ and $ \lambda_0=\lambda_0(\underline{N}_0)>0$
such that for all $N\geq \underline{N}_0$ and $0<\lambda\leq \lambda_0$,  the following statements hold:
\begin{itemize}
  \item There is some semi-algebraic set $\Omega_N=\Omega_N(d,\rho,\gamma, \lambda f,c_1) \subset \mathbb{T}^d$ with  $\deg(\Omega_N)\leq N^{4d}$, and as $\lambda \to \infty$,
  \begin{align*}
    \mathrm{mes}(\T^d\backslash \cap_{N\geq \underline{N}_0} \Omega_N) \to 0.
  \end{align*}
  \item
  If  $\omega\in\Omega_N$ and $E\in\mathbb{R}$, then there exists some  set $X_N=X_N(d,\rho,\gamma, \lambda f,c_1,\omega,E)\subset \mathbb{T}^d$  such that
 \begin{align*}\sup_{1\leq j\leq d,\theta_j^\neg\in \T^{d-1}}\mathrm{mes}(X_N(\theta_j^\neg))\leq e^{-N^{c_1}},\end{align*}
and for  $\theta\notin X_N$, $Q\in\mathcal{E}_N^{0}$, one has
\begin{align*}
 \|G_{Q}(E;\theta)\|&\leq e^{{N}^{\gamma/2}},\\
|G_{Q}(E;\theta)(n,n')|&\leq  e^{-\frac{(1-5^{-\gamma})\rho}{2}|n-n'|^\gamma}\  {\mathrm{for} \ |n-n'|\geq {N}/{10}}.
\end{align*}
\end{itemize}
\end{thm}


\begin{proof}[\bf Proof of Theorem \ref{ldt}]
The proof is based on the multi-scale analysis {scheme} as in \cite{BG,JLS}.  The proof breaks up into  three steps.

{\bf STEP 1: Proof of inductive step}

This will be completed by using semi-algebraic sets {arguments} and Cartan's estimates as in \cite{BG} and \cite{JLS}.

We define for $1\ll N_1\in\N$ the  scales
\begin{align*}
N_2\sim N_1^{2/{c_1}}, \ \log N\sim {N_1^{c_1}}.
\end{align*}
Then we have
\begin{thm}\label{claim}
Let $\Omega_{N_i}$ $(i=1,2)$ be some semi-algebraic set satisfying $\deg(\Omega_{N_i})\leq N_i^{4d}$ and let $\bar\rho_i\in (0,\rho)$. Assume further the following holds:
If  $\omega\in\Omega_{N_i}$ and $E\in\mathbb{R}$, then there exists some semi-algebraic set $X_{N_i}\subset \mathbb{T}^d$ satisfying $\deg(X_{N_i})\leq N_i^{C(d)}$  such that
 \begin{align*}\sup_{1\leq j\leq d,\theta_j^\neg\in \T^{d-1}}\mathrm{mes}( X_{N_i}(\theta_j^\neg))\leq e^{-N_i^{c_1}},\end{align*}
and for  $\theta\notin X_{N_i}$, $Q\in\mathcal{E}_{N_i}^{0}$, one has
\begin{align}
\label{ldt1newi} \|G_{Q}(E;\theta)\|&\leq e^{{N_i}^{\gamma/2}},\\
\label{ldt2newi} |G_{Q}(E;\theta)(n,n')|&\leq  e^{-\bar\rho_i|n-n'|^\gamma}\  {\mathrm{for} \ |n-n'|\geq {N_i}/{10}},\\
\nonumber&(i=1,2).
\end{align}
Then there exist positive constants $c_2<c_3<c_4<\gamma/10$ (depending only on $d$)  such that the following holds:  There exists some semi-algebraic set $\Omega_{{N}} \subset \Omega_{{N}_1}\cap\Omega_{{N}_2}$ with  $\deg(\Omega_{N})\leq {N}^{4d}$ and
$\mathrm{mes}((\Omega_{{N}_1}\cap\Omega_{{N}_2})\backslash\Omega_{{N}})\leq {N}^{-c_2}$
such that, if $\omega\in\Omega_{N}$, then for $E\in \R$ and $\theta\in \T^d$, there is  $\frac{{N}^{c_3}}{10}<M<10{N}^{c_4}$ such that
for all  $k\in \Lambda \backslash \bar\Lambda$,  one has $ \theta+k\omega \mod \Z^d\notin X_{N_1}$, where
\begin{equation*}
  \Lambda=[-M,M]^d,\bar\Lambda=[-M^{\frac{\gamma}{10d}},M^{\frac{\gamma}{10d}}]^{d}.
\end{equation*}
\end{thm}
\begin{proof}
The main point of the proof is to eliminate $(E,\theta)$ by applying Lemmas \ref{svl} and \ref{rmf}. We refer to \cite{BG} for details (see also comments in \cite{JLS}). We remark that the resolvent identity is actually unnecessary in the proof.
\end{proof}

We then construct the $X_N$ by using Cartan's estimates and the resolvent identity.
\begin{lem}[Cartan's estimates, \cite{BB}]\label{mcl}
Let $T(\theta)$ be a self-adjoint $N\times N$ matrix-valued function of a parameter $\theta\in[-\delta,\delta]$  satisfying the following conditions:
\begin{itemize}
\item[(i)] $T(\theta)$ is real analytic in $\theta\in [-\delta,\delta]$ and has a holomorphic extension to
\begin{align*}
\mathcal{D}_{\delta}=\left\{\theta\in\mathbb{C}: \ |\Re \theta|\leq\delta,|\Im{\theta}|\leq \delta\right\}
\end{align*}
satisfying
\begin{align*}
\sup_{\theta\in \mathcal{D}_{\delta}}\|T(\theta)\|\leq K_1, K_1\geq 1.
\end{align*}
\item[(ii)]  For all $\theta\in[-\delta,\delta]$, there is a subset $V\subset [1,N]$ with
\begin{align*}
|V|\leq M
\end{align*}
and
\begin{align*}
\|(R_{[1,N]\setminus V}T(\theta)R_{[1,N]\setminus V})^{-1}\|\leq K_2, K_2\geq 1.
\end{align*}
\item[(iii)]
\begin{align*}
\mathrm{mes}\{\theta\in[-{\delta}, {\delta}]: \ \|T^{-1}(\theta)\|\geq K_3\}\leq 10^{-3}\delta(1+K_1)^{-1}(1+K_2)^{-1}.
\end{align*}
\end{itemize}
Let
\begin{align*}
0<\varepsilon\leq (1+K_1+K_2)^{-10 M}.
\end{align*}
Then
\begin{align}\label{mc5}
\mathrm{mes}\left\{\theta\in\left[-{\delta}/{2}, {\delta}/{2}\right]:\  \|T^{-1}(\theta)\|\geq \varepsilon^{-1}\right\}\leq C\delta e^{-\frac{c\log \frac1\varepsilon}{M\log(K_1+K_2+K_3)}},
\end{align}
where $C, c>0$ are some absolute constants.
\end{lem}

Applying  Cartan's estimates yields the following result.
\begin{thm}\label{thmcar}
Fix $1\leq j\leq d$ and $\theta_j^\neg\in\T^{d-1}$. Write $\theta=(\theta_j, \theta_j^\neg)\in\T^{d}$. Assume that the assumptions of Theorem \ref{claim} are satisfied. Assume further there exist $\widetilde N\in[{N^{c_3}}/{4}, N^{c_4}]$ and $\bar{\Lambda}\subset \Lambda\in \mathcal{E}_{\widetilde N}$ with ${\rm diam}(\bar \Lambda)\leq 4\widetilde N^{\frac{\gamma}{10d}}$ such that,
for any $ k\in  \Lambda \backslash  \bar{\Lambda}$,
there exists some $\mathcal{E}_{N_1}\ni W\subset \Lambda\backslash  \bar{\Lambda}$ such that ${\rm dist}(k,\Lambda \backslash  \bar{\Lambda}\backslash W)\geq {N_1}/{2},$ and $\theta+k\omega\mod\Z^d\notin X_{N_1}$.
Let
\begin{align*}
  Y_\theta=\left\{y\in \R: |y-\theta_j|\leq e^{-10\rho N_1^\gamma},\|G_{\Lambda}(E;(y,\theta_j^\neg))\|\geq e^{{\widetilde N}^{\gamma/2}}\right\}.
\end{align*}
Then for $\omega\in\Omega_{N_1}\cap\Omega_{N_2}$, one has
\begin{align*}
\mathrm{mes}(Y_\theta)\leq e^{-{\widetilde N}^{\gamma/3}}.
\end{align*}

\end{thm}

\begin{proof}
 The proof is similar to that in \cite{JLS}. Let $\mathcal{D}$ be the $e^{-10\rho N_1^\gamma}$-neighbourhood  of $\theta_j$ in the complex plane, i,e.,
\begin{equation*}
 \mathcal{D}=\{y\in \mathbb{C}:\  |\Im y|\leq e^{-10\rho N_1^\gamma}, |\Re y-\theta_j|\leq e^{-10\rho N_1^\gamma}\}.
\end{equation*}
Applying Theorem \ref{claim} yields  for all $k\in \Lambda\backslash \bar{\Lambda}$ and $Q\in \mathcal{E}_{N_1}^0$,
\begin{align}
\|G_{Q}(E;\theta+k\omega)\|&\leq e^{{N_1}^{\gamma/2}},\label{Apr20}\\
|G_{Q}(E;\theta+k\omega)(n,n')|&\leq e^{-\bar{\rho}_1|n-n'|^\gamma}\ {\rm for}\ |n-n'|\geq{N_1}/{10}.\label{Apr21}
\end{align}
Note that for all $n,n'\in [-N_1,N_1]^d$, one has
\begin{align*}
e^{-10\rho N_1^\gamma}<e^{-3\bar\rho_1N_1^\gamma-\bar\rho_1|n-n'|^\gamma}.
\end{align*}
Then by Lemma \ref{pag},  \eqref{Apr20} and \eqref{Apr21}, we have for any $y\in \mathcal{D}$, $Q\in \mathcal{E}_{N_1}^0$
and $k\in\Lambda\backslash\bar\Lambda$,
\begin{align}
\|G_{Q}(E;(\theta_j+y,\theta_j^\neg)+k\omega)\|&\leq 2e^{{N_1}^{\gamma/2}},\label{Apr22}\\
|G_{Q}(E;(\theta_j+y,\theta_j^\neg)+k\omega)(n,n')|&\leq 2e^{-\bar{\rho}_1|n-n'|^\gamma}\ {\rm for}\ |n-n'|\geq{N_1}/10.\label{Apr23}
\end{align}
 Applying  Lemma \ref{res1} with $M_1=M_0=N_1$ implies for any $y\in \mathcal{D}$,
\begin{align}
\label{bs3}\|G_{\Lambda\setminus \bar\Lambda}(E;(\theta_j+y,\theta_j^\neg))\|\leq 4(2N_1+1)^de^{{N_1}^{\gamma/2}}\leq e^{2{N_1}^{\gamma/2}}.
\end{align}

We want to use  Lemma \ref{mcl}  to finish the  proof.
For this purpose, let  \begin{align}\label{sep17}
T(y)=\widetilde{H}_{\Lambda}((\theta_j+y,\theta_j^\neg))-{E},\delta=\delta_1=2e^{-10\rho N_1^\gamma}.
\end{align}
It suffices to verify the assumptions of Lemma \ref{mcl}.
Obviously, $K_1=O(1)$.
By  \eqref{bs3}, one has
\begin{align}\label{mb}
 M=|\bar\Lambda|\leq 100^d\widetilde N^{{\gamma}/{10}},
K_2=e^{2{N_1}^{\gamma/2}}.
\end{align}
Since $\omega\in\Omega_{N_2}$,
\eqref{ldt1newi} and \eqref{ldt2newi} hold at scale $N_2$ for $y$ being outside a set  of measure at most $e^{-{N_2^{c_1}}}$.
Applying Lemma \ref{res1} with $M_0=M_1=N_2$ yields
\begin{eqnarray*}
 \|T^{-1}(y)\|\leq 4(2N_2+1)^de^{{N_2}^{\gamma/2}}\leq e^{2{N_2}^{\gamma/2}}=K_3
\end{eqnarray*}
for $y$ being outside a set  of the measure at most
$$(2\widetilde N+1)^de^{-{N_2^{c_1}}}\leq e^{-{N_2^{c_1}}/{2}}.$$
It follows from $100N_1^\gamma<N_2^{c_1}$ that
\begin{equation*}
10^{-3}\delta_1(1+K_1)^{-1}(1+K_2)^{-1}\geq e^{-{N_2^{c_1}}/{2}}.
\end{equation*}
This verifies (iii) of Lemma \ref{mcl}.
For $\varepsilon=e^{-{\widetilde N}^{\gamma/2}}$,  one has by \eqref{sep17} and \eqref{mb},
  $$\varepsilon<(1+K_1+K_2)^{-10M}.$$
By \eqref{mc5} of Lemma \ref{mcl}, we obtain
\begin{equation*}
\mathrm{mes}(Y_\theta)\leq e^{-\frac{c{\widetilde N}^{\gamma/2}}{N_2\widetilde N^{{\gamma}/{10}}\log \widetilde N}}\leq e^{-\widetilde N^{\gamma/3}}.
\end{equation*}

\end{proof}

Combining Theorems \ref{claim} and \ref{thmcar} yields
\begin{thm}\label{xn}
Let $\omega\in\Omega_N$ and fix $N_\star\in[N,N^2]$. If $E\in\mathbb{R}$ and $c_1<\gamma c_3/10$, then there exists some  set $X_{N_\star}=X_{N_\star}(E,\omega)\subset \mathbb{T}^d$  such that
 \begin{align*}\sup_{1\leq j\leq d,\theta_j^\neg\in \T^{d-1}}\mathrm{mes}( X_{N_\star}(\theta_j^\neg))\leq e^{-N_\star^{c_1}},\end{align*}
and for $\theta\notin X_{N_\star}$, $Q\in\mathcal{E}_{N_\star}^{0}$, one has
\begin{align*}
 |G_{Q}(E;\theta)(n,n')|&\leq  e^{-(\bar\rho_1-\frac{C}{N_1^{\gamma/2}})|n-n'|^\gamma}\  {\mathrm{for} \ |n-n'|\geq {N_\star}/{10}},
\end{align*}
where $C=C(d,\gamma,\rho)>0$.
\end{thm}

\begin{proof}
Fix $1\leq j\leq d,\theta_j^\neg\in\T^{d-1}$ and  $\theta=(\theta_j,\theta_j^\neg)\in\T^d$.  As done in \cite{JLS} by using Theorem \ref{claim}, for such $\theta$ and any $n\in Q\in\mathcal{E}_{N_\star}^0$, there exist $\frac{1}{4}N^{c_3}\leq \widetilde{N}_{n,\theta} \leq N^{c_4}$,  $\Lambda_{n,\theta} \in \mathcal{E}_{\widetilde{N}}$ and $\bar{\Lambda}_{n,\theta}$, such that
\begin{align*}
n\in \bar\Lambda_{n,\theta}\subset\Lambda_{n,\theta} \subset Q, {\rm dist}(n,  Q\backslash \Lambda_{n,\theta})\geq{\widetilde{N}}/{2},\ {\rm diam}(\bar{\Lambda}_{n,\theta}) \leq 4 \widetilde{N}_{n,\theta}^{\frac{\gamma}{10d}}.
\end{align*}
Moreover,  for any $ k\in  \Lambda_{n,\theta}\backslash  \bar{\Lambda}_{n,\theta}$, we have $\theta+k\omega\mod\Z^d\notin X_{N_1}$,  and
there exists some $\mathcal{E}_{N_1} \ni W \subset \Lambda_{n,\theta}\backslash  \bar{\Lambda}_{n,\theta}$ such that
\begin{equation*}\label{Apr5}
  k\in W,\  {\rm dist}(k,\Lambda_{n,\theta}\backslash  \bar{\Lambda}_{n,\theta}\backslash W)\geq {N_1}/{2}.
\end{equation*}

We now fix the above $\widetilde N_{n,\theta}, \bar\Lambda_{n,\theta},\Lambda_{n,\theta}$ throughout the set $\{(y,\theta_j^\neg)\in\R^d:\ |y-\theta_j|\leq e^{-10\rho N_1^\gamma}\}$. Recalling Lemma \ref{pag} and the above constructions, the assumptions of Theorem \ref{thmcar} are essentially satisfied. Applying Theorem \ref{thmcar} implies that 
 there exists a set $Y_{n,\theta}\subset \{y\in\R:\ |y-\theta_j|\leq e^{-10\rho N_1^\gamma}\}$ such that
\begin{equation}\label{Apr6}
\mathrm{mes}(Y_{n,\theta})\leq e^{-\widetilde{N}_{n,\theta}^{{\gamma}/{3}}},
\end{equation}
and for $\theta_j\notin Y_{n,\theta}$, one has
\begin{equation*}\label{Apr7}
\|G_{\Lambda_{n,\theta}}(E;\theta)\|\leq e^{{\widetilde{N}_{n,\theta}^{\gamma/2}}}.
\end{equation*}
Applying Lemma \ref{res2} with $M_0=N_1, \Lambda=\Lambda_{n,\theta}$ and $\Lambda_1=\bar\Lambda_{n,\theta}$ yields 
\begin{equation*}\label{Apr8}
  |G_{\Lambda_{n,\theta}}(E;\theta)(n,n')|\leq  e^{- (\bar{\rho}-\frac{C}{N_1^{\gamma/2}})|n-n'|^\gamma}\  {\mathrm{for} \ |n-n'|\geq {\widetilde{N}_{n,\theta}}/{10}}.
\end{equation*}

Cover $[0,1]$ by pairwise disjoint $e^{-10\rho N_1^\gamma}$-size intervals and
let
\begin{equation}\label{Apr9}
 X_{N_\star}(\theta_j^\neg)=\bigcup_{Q\in\mathcal{E}_{N_\star}^0, n\in Q, \theta=(\theta_j,\theta_j^\neg)}
 Y_{n,\theta}.
\end{equation}
We remark that while $\theta=(\theta_j,\theta_j^\neg)$ varies on a line for a fixed $\theta_j^\neg$, the total number of $Y_{n,\theta}$ is bounded by $e^{10\rho N_1^\gamma}$. Thus by \eqref{Apr6}, \eqref{Apr9} and $c_1<\gamma c_3/10$,
one has 
\begin{equation*}\label{Apr12}
\mathrm{mes}({X}_{{N_\star}}(\theta_j^\neg))\leq C(2N+1)^{d}e^{10\rho N_1^\gamma}e^{-\widetilde N_{n,\theta}^{\gamma/3}} \leq e^{-{N_\star}^{c_3\gamma/7}}\leq e^{-{N_\star}^{c_1}}.
\end{equation*}

Suppose  now $\theta\notin X_{N_\star}$. Applying Lemma \ref{res1} with $\Lambda=Q\in \mathcal{E}_{N_\star}^0$, $M_0=\frac{1}{4}N^{c_3}$ and $M_1=\widetilde N_{n,\theta}\leq N^{c_4}$, 
 one has
\begin{equation*}\label{Apr10}
\|G_{Q}(E;\theta)\|\leq 4(2 N^{c_4}+1)^de^{{N^{c_4\gamma/2}}}\leq e^{{N_\star}^{\gamma/2}}.
\end{equation*}
Applying  Lemma \ref{res2} with $\Lambda=Q$, $M_0=\frac{1}{4}N^{c_3}$, $M_1= \widetilde N_{n,\theta}\leq N^{c_4}$ and $ \Lambda_1=\emptyset$,
we have
\begin{equation*}\label{Apr11}
  |G_{Q}(E;\theta)(n,n')|\leq  e^{- (\bar{\rho}_1-\frac{C}{N_1^{\gamma/2}})|n-n'|^\gamma}\  {\mathrm{for} \ |n-n'|\geq {N_\star}/{10}}.
\end{equation*}
This proves the {theorem}.
\end{proof}

{\bf STEP 2: Proof of initial step}

\begin{lem}\label{lemini}
Let 
\begin{align*}
X_{N}=\bigcup_{|n|\leq {N}}\left\{\theta:\  |f(\theta+n\omega)-E|<\delta\right\}.
\end{align*}
Then  we have for any $1\leq j\leq d$,
\begin{align*}
\sup_{\theta_j^\neg\in\T^{d-1}}\mathrm{mes}(X_N(\theta_j^\neg))\leq C(2N+1)^d\delta^{c},
\end{align*}
where $C=C(f)>0,c=c(f)>0$.
Moreover, if
$\lambda^{-1}\geq 2\delta^{-1}(2N+1)^d$,
then for any $\theta\notin X_{N},\ \omega\in\T^d$ and $\Lambda\subset [-N,N]^d$, we have
\begin{align*}
\|G_{\Lambda}(E;\theta)\|&\leq 2\delta^{-1}, \\
|G_{\Lambda}(E;\theta)(n,n')|&\leq 2\delta^{-1}e^{-{\rho}|n-n'|^{\gamma}}.
\end{align*}
\end{lem}
\begin{proof}
The measure bound follows  from a {\L}ojasiewicz  type inequality (see Lemma 5.2 of \cite{JLS}) and the \textit{non-degeneracy} condition of $f$ immediately.

The Green's function estimates follow from the Neumann series argument. For details, we refer to \cite{JLS} (or the proof of Lemma \ref{pag}, which deals with some more complicated setting).

\end{proof}

{\bf STEP 3: Completion of the proof}

This will follow from Theorem \ref{xn}, Lemma \ref{lemini} and a multi-scale induction. For details, we refer to \cite{JLS}.
\end{proof}

\section{Proof of Theorem \ref{mthm2}}
The key {point} of the proof is to eliminate the energy $E$ in the LDT and this needs to remove further $\omega$  by semi-algebraic geometry arguments (i.e. Lemma \ref{proj}).

\begin{proof}[\bf Proof of Theorem \ref{mthm2}]
The proof is rather standard and based on Theorems \ref{ldt}, \ref{claim} and Lemma \ref{proj}. We refer to \cite{BG} for details.
\end{proof}

\section{Proof of Theorem \ref{mthm1}}
In this section we will prove Theorem \ref{mthm1} by using the  LDT and the Delyon's trick \cite{Del87}. 

Fix
\begin{align*}
\bar\rho=(1-5^{-\gamma})\rho.
\end{align*}
We have the  Poisson's identity: For $\widetilde H(\theta)\xi=E\xi$ and $n\in \Lambda\subset\Z^d$,
\begin{align}\label{pi}
\xi_n=-\lambda\sum_{n'\in\Lambda,n''\notin\Lambda}G_{\Lambda}(E;\theta)(n,n')\widehat{v}_{n'-n''}\xi_{n''}.
\end{align}
\begin{proof}[\bf Proof of Theorem \ref{mthm1}]
Let $\omega\in \bigcap_{N\geq \underline{N}_0} \Omega_{N}$ and $ 0<\lambda\leq\lambda_0$ be as in Theorem \ref{ldt}. 
Suppose $H_{\lambda v,\omega,x}$ has some eigenvalue $E$. Then there must be some $0\neq\psi=\{\psi_\ell\}_{\ell\in\Z}\in\ell^2(\Z)$ so that
\begin{align*}
\sum_{\ell'\in\Z}\widehat g_{\ell-\ell'}\psi_{\ell'}+(\lambda v({x}+\ell{\omega})-E)\psi_\ell=0.
\end{align*}
Define
\begin{align*}
F(\theta)&=\sum_{\ell\in\Z}\psi_\ell e^{2\pi i \ell\theta}
\end{align*}
and
\begin{align*}
\xi_{n}(\theta)&=e^{2\pi i{n}\cdot {x}}F(\theta+{n}\cdot\omega).
\end{align*}
We have
\begin{align}\label{fn0}
\|F\|_{L^2(\T)}=\|\psi\|_{\ell^2(\Z)}>0
\end{align}
and by direct computation
\begin{align}\label{Fe}
(g(\theta)-E)F(\theta)+\lambda\sum_{{k}\in\Z^d}\widehat v_{k} \xi_{k}(\theta)=0.
\end{align}

Then
\begin{align*}
\int_{\T}\sum_{{n}\in\Z^d}\frac{|\xi_{n}(\theta)|^2}{1+|{n}|^{2d}}{d}\theta&=\sum_{{n}\in\Z^d}\frac{\|F\|_{L^2(\T)}^2}{1+|{n}|^{2d}}\\
&\leq C\|F\|_{L^2(\T)}^2<\infty.
\end{align*}
This implies that for a.e. $\theta$, we have $\sum\limits_{{n}\in\Z^d}\frac{|\xi_{n}(\theta)|^2}{1+|{n}|^{2d}}<\infty$ and
\begin{align*}
|\xi_{n}(\theta)|\leq C(\theta,d)|{n}|^{d},\ C(\theta,d)>0.
\end{align*}
We let $\theta=\theta+{n}\cdot\omega$  in (\ref{Fe}).  Then
\begin{align*}
(g(\theta+{n}\cdot\omega)-E)F(\theta+{n}\cdot\omega)+\lambda\sum_{{k}\in\Z^d}\widehat v_{k} e^{2\pi i{k}\cdot {x}}F(\theta+({n}+{k})\cdot\omega)=0.
\end{align*}
Multiplying by $e^{2\pi i{n}\cdot {x}}$ on the above equality implies
\begin{align}\label{xie}
(g(\theta+{n}\cdot\omega)-E)\xi_{n}(\theta)+\lambda\sum_{{k}\in\Z^d}\widehat v_{{n}-{k}} \xi_{{k}}(\theta)=0.
\end{align}

Now let  $X_N=X_N(\omega,E)$ be as in Theorem \ref{ldt}. We define
\begin{align*}
\Theta=\bigcup_{M\geq \underline{N}_0}\bigcap_{N\geq M} X_N.
\end{align*}
Then by $\mathrm{mes}(X_N)\leq e^{-N^{c_1}}$, one has $\mathrm{mes}(\Theta)=0$. Fix $\theta\in \T\setminus\Theta$. Then there exists  $M\geq \underline{N}_0$  such that
$$\theta\notin X_{N}\ \mathrm{for}\ N\geq M.$$
Recalling \eqref{pi},  \eqref{xie} and Theorem \ref{ldt}, one has for $N\geq M\gg1$,
\begin{align*}
|F(\theta)|=|\xi_{0}(\theta)|&=|\sum_{|{n}|\leq N,|{n}'|>N}G_{[-N,N]^d}(E;\theta)(0,{n})\widehat v_{{n}-{n}'} \xi_{{n}'}(\theta)|\\
&\leq C(\theta,d) \sum_{|{n}|\leq N,|{n}'|>N}e^{-\frac{\bar\rho}{2}|{n}|^\gamma+\frac{\bar\rho}{2} (N/10)^\gamma+N^{{\gamma}/{2}}}e^{-\rho|{n}-{n}'|^\gamma} |{n}'|^d\\
&\leq C(\theta,d)N^d\sum_{|{n}'|>N}e^{-\frac{\bar\rho}{2}|{n}'|^\gamma+\frac{\bar\rho}{2} (N/10)^\gamma+N^{\gamma/2}}|{n}'|^d\\
&=o(N).
\end{align*}
Letting $N\to\infty$, we have $F(\theta)=0$ for a.e. $\theta\in\T\setminus\Theta$.  Thus $\|F\|_{L^2(\T)}=0$, which contradicts (\ref{fn0}).

This proves Theorem \ref{mthm1}.
\end{proof}


\section{Proof of Theorem \ref{cor1}}
In this section we will prove Theorem \ref{cor1} by applying the LDT. The main idea of the proof  is from Bourgain \cite{BB}, where the $1D$ analytic Schr\"odinger operator with the single-frequency  was investigated. For $f(\theta+n\omega)=g(\theta+n\cdot\omega)$, we denote by $\widetilde\Sigma$ the spectrum of $\widetilde H_{\lambda f,\omega,\theta}$, which is also independent of $\theta$.  Thus to prove Theorem \ref{cor1}, it suffices to show  $\widetilde\Sigma$ has  positive Lebesgue measure.

For simplicity, we write $\widetilde H(\theta)=\widetilde H_{\lambda f,\omega,\theta}$ and $$\widetilde H_N(\theta)=R_{\Lambda}\widetilde H(\theta)R_{\Lambda} \ \mathrm{for}\ \Lambda\in \mathcal{E}_N^0.$$ We denote by $\{e_k:\ k\in\Z^d\}$ (resp. $\langle\cdot,\cdot\rangle$) the standard orthogonal basis (resp. inner product) on $\ell^2{(\Z^d)}$.
\begin{lem}\label{l61}
Let $\omega\in\bigcap_{N\geq \underline{N}_0}\Omega_{N}$ and $N_0\gg \underline{N}_0$. Then there exists a positive constant $\lambda_0=\lambda_0(N_0)\ll1$ such that the following holds: If $0\leq \lambda\leq\lambda_0$, then there exist an interval $I_0\subset[0,1]$ and a continuous function $E_{I_0}(\cdot)$ on $I_0$ satisfying
\begin{align*}
|I_0|\geq N_0^{-C_1}
\end{align*}
and for $\theta\in I_0$,
\begin{align*}
\min_{\xi\in\mathrm{Span}\{e_k:\ k\in\Z^d, |k|\leq N_0\},\ \|\xi\|=1}\|(\widetilde H(\theta)-E_{I_0}(\theta))\xi\|\leq e^{-c_5(\log N_0)^{\gamma/c_1}},
\end{align*}
where $0<c_5=c_5(\gamma,\rho)\ll 1$ and $C_1=C_1(d)>1$.
\end{lem}
\begin{proof}
Fix any $\theta$. Denote by $\lambda_s(\theta),\  1\leq s\leq (2N_0+1)^d$ (resp. $\phi_s,\ \|\phi_s\|=1$) the eigenvalues (resp. corresponding eigenvectors) of $\widetilde H_{N_0}(\theta)$, where $N_0\gg 1$ will be specified later.
Then one has
\begin{align}\label{e60}
e_0=\sum_{1\leq s\leq (2N_0+1)^d}\langle e_0,\phi_s\rangle\phi_s.
\end{align}
Obviously, we have
\begin{align*}
\|(\widetilde H(\theta)-f(\theta))e_0\|\leq \sum_{m\in\Z^d}\lambda e^{-\rho|m|^\gamma}\leq C(\rho,\gamma,d)\lambda.
\end{align*}
Thus
\begin{align}
\nonumber(\widetilde H_{N_0}(\theta)-f(\theta))e_0&=(\widetilde H(\theta)-f(\theta))e_0-R_{\Z^d\setminus [-N_0,N_0]^d}\widetilde H(\theta)e_0\\
\label{e61}&=O(\lambda).
\end{align}
On the other hand, we have
\begin{align}
\nonumber\widetilde H_{N_0}(\theta)e_0&=\sum_{1\leq s\leq (2N_0+1)^d}\langle e_0,\phi_s\rangle \widetilde H_{N_0}(\theta)\phi_s\\
\label{e62}&=\sum_{1\leq s\leq (2N_0+1)^d}\langle e_0,\phi_s\rangle  \lambda_s(\theta)\phi_s.
\end{align}
Thus by combining \eqref{e60}, \eqref{e61} and \eqref{e62}, we obtain
\begin{align}
\nonumber&\left(\sum_{1\leq s\leq (2N_0+1)^d}|\langle e_0,\phi_s\rangle|^2|\lambda_s(\theta)-f(\theta)|^2\right)^{1/2}\\
\nonumber&=\left\|\sum_{1\leq s\leq (2N_0+1)^d}\langle e_0,\phi_s\rangle(\lambda_s(\theta)-f(\theta))\phi_s\right\|\\
\label{e63}&\leq C\lambda.
\end{align}
Since $1=\|e_0\|^2=\sum\limits_{1\leq s\leq (2N_0+1)^d}|\langle e_0,\phi_s\rangle|^2$, there exists some $s_\star\in [1, (2N_0+1)^d]$ so that
\begin{align}\label{e64}
|\langle e_0,\phi_{s_\star}\rangle|\geq {(2N_0+1)^{-d/2}},
\end{align}
which together with \eqref{e63} implies
\begin{align}\label{e65}
|\lambda_{s_\star}(\theta)-f(\theta)|\leq C{(2N_0+1)^{d/2}}\lambda.
\end{align}
Recall that $\omega\in\Omega_{N_0}, N_0\gg \underline{N}_0$. We have by Theorem \ref{claim},  there exist $M_0\sim (\log N_0)^{1/c_1}\geq \underline{N}_0$ and $N_0^{c_3}/10\leq M_1\leq 10N_0^{c_4}$ so that $\theta+n\omega\mod\Z^d \notin X_{M_0}$ for all $n$ satisfying $$N_0^{c_3}/10\leq M_1^{\gamma/(10d)}\leq |n|\leq M_1\leq 10N_0^{c_4}.$$
Fix $M_1^{\gamma/(10d)}\leq |n|\leq M_1$. Then we can find $Q(n)\in\mathcal{E}_{M_0}$ so that
\begin{align*}
\mathrm{dist} (n,\Lambda\setminus\Lambda_1\setminus Q(n))&\geq M_0/2,\\
\|G_{Q(n)}(E;\theta)\|&\leq e^{M_0^{\gamma/2}},\\
|G_{Q(n)}(E;\theta)(k,k')|&\leq e^{-\frac{\bar\rho}{2}|k-k'|^\gamma}\ \mbox{for $|k-k'|\geq M_0/10$}.
\end{align*}
Thus by the Poisson's identity \eqref{pi}, we have for $M_0\geq M_0(\gamma,\bar\rho)\gg1$ and $\|\phi_{s_\star}\|=1$,
\begin{align}
\nonumber|\phi_{s_\star}(n)|&=\left|\sum_{n_1\in Q(n),n_2\in \Lambda\setminus\Lambda_1\setminus Q(n)}\lambda G_{Q(n)}(E;\theta)(n,n_1)\widehat{v}_{n_1-n_2}\phi_{s_\star}(n_2)\right|\\
\nonumber&\leq \sum_{n_1\in Q(n),n_2\in \Lambda\setminus\Lambda_1\setminus Q(n)}e^{M_0^{\gamma/2}+\frac{\bar\rho}{2}({M_0}/10)^\gamma-\frac{\bar\rho}{2}|n-n_1|^\gamma-\rho|n_1-n_2|^\gamma} \\
\nonumber&\leq \sum_{n_1\in Q(n),n_2\in \Lambda\setminus\Lambda_1\setminus Q(n)}e^{M_0^{\gamma/2}+\frac{\bar\rho}{2}({M_0}/10)^\gamma-\frac{\bar\rho}{2}|n-n_2|^\gamma} \\
\nonumber&\leq \sum_{n_1\in Q(n),n_2\in \Lambda\setminus\Lambda_1\setminus Q(n)}e^{M_0^{\gamma/2}+\frac{\bar\rho}{2}({M_0}/10)^\gamma-\frac{\bar\rho}{2}(M_0/2)^\gamma} \\
\label{e66}&\leq e^{-c(\log N_0)^{\gamma/c_1}}.
\end{align}
We define
\begin{align*}
J=[{M_1+M_1^{\gamma/(10d)}}/{2}],\ \Lambda=[-J,J]^d\subset [-N_0,N_0]^d.
\end{align*}
Then by \eqref{e64}, $$\|R_{\Lambda}\phi_{s_\star}\|\geq (2N_0+1)^{-d/2}.$$ Define now $$\psi=\frac{R_{\Lambda}\phi_{s_\star}}{\|R_{\Lambda}\phi_{s_\star}\|}.$$

Since $(\widetilde H_{N_0}(\theta)-\lambda_{s_{\star}}(\theta))\phi_{s_\star}=0$, we have
\begin{align}\label{kf}
R_{\Lambda}(\widetilde H(\theta)-\lambda_{s_\star}(\theta))\psi=-\|R_{\Lambda}\phi_{s_\star}\|^{-1}R_{\Lambda}\widetilde H(\theta)R_{[-N_0,N_0]^d\setminus\Lambda}\phi_{s_\star}.
\end{align}
Thus by direct computations, we obtain
\begin{align*}
(\widetilde H(\theta)-\lambda_{s_\star}(\theta))\psi&=R_{\Z^d\setminus[-N_0,N_0]^d}\widetilde H(\theta)\psi+R_{[-N_0,N_0]^d\setminus\Lambda}\widetilde H(\theta)\psi\\
&\ \ +R_{\Lambda}(\widetilde H(\theta)-\lambda_{s_\star}(\theta))\psi\\
&=\|R_{\Lambda}\phi_{s_\star}\|^{-1}R_{\Z^d\setminus[-N_0,N_0]^d}\widetilde H(\theta)R_\Lambda\phi_{s_\star}\\
&\ \ +\|R_{\Lambda}\phi_{s_\star}\|^{-1}R_{[-N_0,N_0]^d\setminus\Lambda}\widetilde H(\theta)R_\Lambda\phi_{s_\star}\\
&\ \ + (-\|R_{\Lambda}\phi_{s_\star}\|^{-1}R_{\Lambda}\widetilde H(\theta)R_{[-N_0,N_0]^d\setminus\Lambda}\phi_{s_\star})\ (\mbox{by \eqref{kf}})\\
&=(I)+(II)+(III).
\end{align*}
For $(I)$, we have
\begin{align}
\nonumber\|(I)\|^2&\leq \lambda^2(2N_0+1)^d\sum_{|m|>N_0}\left(\sum_{|n|\leq J}e^{-\rho|m-n|^\gamma}\right)^2\\
\nonumber&\leq \lambda^2(2N_0+1)^{2d}e^{2\rho J^\gamma}\left(\sum_{|m|> N_0}e^{-\rho|m|^\gamma}\right)^2\\
\label{e67}&\leq e^{-\rho N_0^{\gamma}} \ (\mbox{since $J\leq 10N_0^{c_4}$}).
\end{align}
For $(II)$, we have since  \eqref{e66},
\begin{align}
\nonumber\|(II)\|^2&\leq \lambda^2{(2N_0+1)^d}\sum_{J<|m|\leq N_0}\left|\sum_{|n|\leq J}e^{-\rho|m-n|^\gamma}\phi_{s_\star}(n)\right|^2\\
\nonumber&\leq \lambda^2{(2N_0+1)^d}\sum_{J<|m|\leq N_0}\left(\sum_{|n|\leq M_1^{\gamma/(10d)}}e^{-\rho|m-n|^\gamma}\right)^2 \\
\nonumber&\ \ +\lambda^2{(2N_0+1)^d}\sum_{J<|m|\leq N_0}\left(\sum_{M_1^{\gamma/(10d)}\leq |n|\leq J} e^{-c(\log N_0)^{\gamma/c_1}}\right)^2\\
\nonumber&\leq \lambda^2{(10N_0)}^{3d}e^{-cJ^{\gamma}}+\lambda^2{(10N_0)}^{3d} e^{-c(\log N_0)^{\gamma/c_1}}\\
\label{e68}&\leq e^{-3c_5(\log N_0)^{\gamma/c_1}}.
\end{align}
Similarly, for $(III)$, we have
\begin{align}
\nonumber\|(III)\|^2&\leq \lambda^2{(2N_0+1)^d}\sum_{|m|\leq J}\left|\sum_{J\leq |n|\leq N_0}e^{-\rho|m-n|^\gamma}\phi_{s_\star}(n)\right|^2\\
\nonumber&\leq \lambda^2{(2N_0+1)^d}\sum_{|m|\leq J}\left(\sum_{J\leq |n|\leq M_1}e^{-c(\log N_0)^{\gamma/c_1}}\right)^2 \\
\nonumber&\ \ +\lambda^2{(2N_0+1)^d}\sum_{|m|\leq J}\left(\sum_{M_1\leq |n|\leq N_0}e^{-\rho|m-n|^\gamma} \right)^2\\
\label{e68new}&\leq e^{-3c_5(\log N_0)^{\gamma/c_1}}.
\end{align}

Thus combining \eqref{e67},  \eqref{e68} and \eqref{e68new},  we obtain
\begin{align*}
\min_{\xi\in\mathrm{Span}\{e_k:\ k\in\Z^d, |k|\leq J\},\ \|\xi\|=1}\|(\widetilde H(\theta)-\lambda_{s_\star}(\theta))\xi\|\leq e^{-c_5(\log N_0)^{\gamma/c_1}},
\end{align*}
or equivalently
\begin{align}\label{e69}
\|(R_{\Lambda}(\widetilde H(\theta)-\lambda_{s_\star}(\theta))^{*}(\widetilde H(\theta)-\lambda_{s_\star}(\theta))R_{\Lambda})^{-1}\|\geq e^{2c_5(\log N_0)^{\gamma/c_1}}.
\end{align}

Define for $1\leq s_\star\leq (2N_0+1)^d$ and $J\in [N_0^{c_3}/10,N_0^{c_4}]$ the set $\Gamma_{s_\star,J}\subset [0,1]$ of $\theta$ for which \eqref{e65} and \eqref{e69} hold. It is well-known that $\lambda_{s_\star}(\theta)$ is Lipschitz continuous in $f$ (see \cite{Tao} for details). By a standard truncation argument, we can replace $f(\theta)$ by a polynomial in $\theta$ of degree $CN_0^2$. Notice that the $\lambda_{s_\star}(\theta)$  {satisfies} the equation
\begin{align*}
\zeta^{D}+\sum_{r<D}c_r(\theta)\zeta^D=0,
\end{align*}
where $D=(2N_0+1)^d$  and $c_r(\theta)$ are polynomials of degree at most $N_0^C$. Expressing \eqref{e69} by the Cramer's rule, a polynomial condition
\begin{align*}
P(\theta,\zeta)>0
\end{align*}
is obtained in $(\theta,\zeta=\lambda_{s_\star}(\theta))$. Recalling Lemmas \ref{tsp} and \ref{btb}, $\Gamma_{s_\star,J}$ can be decomposed into $N_0^C$ many intervals $I'\subset\Gamma_{s_\star,J}$. For each such $I'$, we set $E_{I'}(\theta)=\lambda_{s_\star}(\theta),\theta\in I'$. Let $\mathcal{F}_0$ be the collection of all such intervals $I'$ (counting all possible $s_\star,J$). Then $\#\mathcal{F}_0\leq N_0^{C_1}$. In particular, for $\theta\in I'\subset\Gamma_{s_\star,J}$, we have
\begin{align*}
\min_{\xi\in\mathrm{Span}\{e_k:\ k\in\Z^d, |k|\leq N_0\},\ \|\xi\|=1}\|(\widetilde H(\theta)-E_{I'}(\theta))\xi\|\leq e^{-c_5(\log N_0)^{\gamma/c_1}}.
\end{align*}
We observe that
\begin{align*}
[f_{\min},f_{\max}]&=\bigcup_{s_\star,J} f(\Gamma_{s_\star,J})\\
&\subset \bigcup_{s_\star,J} \bigcup_{I'\subset\Gamma_{s_\star,J} }\left(\lambda_{s_\star}(I')+[-C{N_0^{d/2}}\lambda,C{N_0^{d/2}}\lambda]\right)\ (\mbox{by \eqref{e65}})\\
&=\bigcup_{I'\in\mathcal{F}_0 }\left(E_{I'}(I')+[-C{N_0^{d/2}}\lambda,C{N_0^{d/2}}\lambda]\right)
\end{align*}
Thus for $N_0\gg \underline{N}_0$ and $\lambda\leq \lambda_0(N_0)\ll1$, we get
\begin{align*}
0<f_{\max}-f_{\min}&\leq \mathrm{mes}\left(\bigcup_{I'\in\mathcal{F}_0}E_{I'}(I')\right)+N_0^{C_1}\lambda\\
&\leq \mathrm{mes}\left(\bigcup_{I'\in\mathcal{F}_0 }E_{I'}(I')\right)+\sqrt{\lambda}.
\end{align*}

Define $I_0$ to be the interval in $\mathcal{F}_0$ with the \textit{maximal} length. Then by $[0,1]\subset\bigcup_{I'\in\mathcal{F}_0}I'$ and $\#\mathcal{F}_0\leq N_0^{C_1}$, we obtain $|I_0|\geq N_0^{-C_1}$. If $\theta\in I_0$, we have
\begin{align*}
\min_{\xi\in\mathrm{Span}\{e_k:\ k\in\Z^d, |k|\leq N_0\},\ \|\xi\|=1}\|(\widetilde H(\theta)-E_{I_0}(\theta))\xi\|\leq e^{-c_5(\log N)^{\gamma/c_1}}.
\end{align*}
This proves the lemma.
\end{proof}

The following lemma is an inductive extension of Lemma \ref{l61}.
\begin{lem}\label{l62}
Let $\gamma/c_1>100$. Let $I\subset[0,1]$ be an interval and $E(\theta)\in\sigma(\widetilde H_{N}(\theta))$ be a continuous function on $I$. Assume again that
\begin{align}\label{e621}
\min_{\xi\in\mathrm{Span}\{e_k:\ k\in\Z^d, |k|\leq N\},\ \|\xi\|=1}\|(\widetilde H(\theta)-E(\theta))\xi\|\leq e^{-c_5(\log N)^{\gamma/c_1}},
\end{align}
where $c_5>0$ is given by Lemma \ref{l61}.

Let
\begin{align}\label{e622}
\N \ni N_1\sim e^{(\log N)^{10}}.
\end{align}
Then there exists a system $(I', E_{I'}(\cdot))_{I'\in \mathcal{F}_1}$ such that the following holds:  $\mathcal{F}_1$ is a collection of at most $N_1^{C_1}$ intervals $I'\subset I$ so that  $E_{I'}(\theta)\in\sigma(\widetilde H_{N_1}(\theta))$ is a continuous function on $I'$,  and for $\theta\in I'$,
\begin{align}\label{e623}
\min_{\xi\in\mathrm{Span}\{e_k:\ k\in\Z^d, |k|\leq N_1\},\ \|\xi\|=1}\|(\widetilde H(\theta)-E_{I'}(\theta))\xi\|\leq e^{-c_5(\log N_1)^{\gamma/c_1}}.
\end{align}
Moreover,
\begin{align}\label{e624}
\mathrm{mes}\left(\bigcup_{I'\in\mathcal{F}_1 }E_{I'}(I')\right)\geq \mathrm{mes}(E(I))-\frac{1}{N_1}.
\end{align}
\end{lem}
\begin{proof}
The proof is similar to that of Lemma \ref{l61}. Fix a $\theta\in I$. Choose a $\xi$ with $\|\xi\|=1$ and $\xi\in \mathrm{Span}\{e_k:\ k\in\Z^d, |k|\leq N\}$ so that \eqref{e621} holds.
Denote by $\lambda_s(\theta),\ 1\leq s\leq (2N_1+1)^d$ (resp. $\phi_s,\ \|\phi_s\|=1$) the eigenvalues (resp. corresponding eigenvectors ) of $\widetilde H_{N_1}(\theta)$.
Then one has
\begin{align*}
\xi=\sum_{1\leq s\leq (2N_1+1)^d}\langle \xi,\phi_s\rangle\phi_s.
\end{align*}
Obviously, we have
\begin{align*}
\|(\widetilde H(\theta)-E(\theta))\xi\|\leq e^{-c_5(\log N)^{\gamma/c_1}}.
\end{align*}
Thus
\begin{align*}
\|\nonumber(\widetilde H_{N_1}(\theta)-E(\theta))\xi\|&=\|(\widetilde H(\theta)-E(\theta))\xi-R_{\Z^d\setminus [-N_1,N_1]^d}\widetilde H(\theta)\xi\|\\
&\leq 2e^{-c_5(\log N)^{\gamma/c_1}}.
\end{align*}
On the other hand, one has
\begin{align*}
\widetilde H_{N_1}(\theta)=\sum_{1\leq s\leq (2N_1+1)^d}\langle \xi,\phi_s\rangle  \lambda_s(\theta)\phi_s.
\end{align*}
Thus
\begin{align*}
&\left(\sum_{|s|\leq N_1}|\langle \xi,\phi_s\rangle|^2|\lambda_s(\theta)-E(\theta)|^2\right)^{1/2}\leq 2e^{-c_5(\log N)^{\gamma/c_1}}.
\end{align*}
Since $\|\xi\|=1$, there exists a $s_\star\in [1, (2N_1+1)^d]$ so that
\begin{align*}
|\langle \xi,\phi_{s_\star}\rangle|\geq (2N_1+1)^{-d/2}
\end{align*}
and
\begin{align}\label{e625}
|\lambda_{s_\star}(\theta)-E(\theta)|\leq 2{(2N_1+1)}^{d/2}e^{-c_5(\log N)^{\gamma/c_1}}.
\end{align}
As in the proof of Lemma \ref{l61}, we have for some $M_1\in[N_1^{c_3}/10,10N_1^{c_4}]$,
\begin{align*}
|\phi_{s_\star}(n)|\leq e^{-c(\log N_1)^{\gamma/c_1}}\ \mbox{for $M_1^{\gamma/(10d)}\leq |n|\leq M_1$}.
\end{align*}
Note that  for $J=[(M_1^{\gamma/(10d)}+M_1)/2]$ and $\Lambda=[-J,J]^d$, one has $\|R_{\Lambda}\phi_{s_\star}\|\geq (2N_1+1)^{-d/2}$. Define $$\psi=\frac{R_{\Lambda}\phi_{s_\star}}{\|R_{\Lambda}\phi_{s_\star}\|}.$$ Similar to the proof of Lemma \ref{l61}, we have
\begin{align*}
\min_{\xi\in\mathrm{Span}\{e_k:\ k\in\Z^d, |k|\leq J\},\ \|\xi\|=1}\|(\widetilde H(\theta)-\lambda_{s_\star}(\theta))\xi\|\leq e^{-c_5(\log N_1)^{\gamma/c_1}},
\end{align*}
or equivalently
\begin{align}\label{e626}
\|(R_{\Lambda}(\widetilde H(\theta)-\lambda_{s_\star}(\theta))^{*}(\widetilde H(\theta)-\lambda_{s_\star}(\theta))R_{\Lambda})^{-1}\|\geq e^{2c_5(\log N_1)^{\gamma/c_1}}.
\end{align}

Similarly, we define for $1\leq s_\star\leq (2N_1+1)^d$ and $N_1^{c_3}/10\leq J\leq 10N_1^{c_4}$ the set $\Gamma_{s_\star,J}\subset I$ of $\theta$ for which \eqref{e625} and \eqref{e626} hold. Using semi-algebraic sets arguments as previous, $\Gamma_{s_\star,J}$ can be decomposed into $N_1^C$ many intervals $I'\subset\Gamma_{s_\star,J}$. For each such $I'$, we set $E_{I'}(\theta)=\lambda_{s_\star}(\theta),\theta\in I'$. Let $\mathcal{F}_1$ be the collection of all such intervals $I'$ (counting all possible $s_\star,J$). Then $\#\mathcal{F}_1\leq N_1^{C_1}$. In particular, for $\theta\in I'\subset\Gamma_{s_\star,J}$, we have
\begin{align*}
\min_{\xi\in\mathrm{Span}\{e_k:\ k\in\Z^d, |k|\leq N_1\},\ \|\xi\|=1}\|(\widetilde H(\theta)-E_{I'}(\theta))\xi\|\leq e^{-c_5(\log N_1)^{\gamma/c_1}}.
\end{align*}
This proves \eqref{e623}.
Observe that again since \eqref{e625},
\begin{align*}
E(I)&=\bigcup_{s_\star,J} E(\Gamma_{s_\star,J})\\
&\subset \bigcup_{s_\star,J} \bigcup_{I'\subset\Gamma_{s_\star,J} }\left(\lambda_{s_\star}(I')+[-2{(2N_1+1)}^{d/2}e^{-c_5(\log N)^{\gamma/c_1}},2{(2N_1+1)}^{d/2}e^{-c_5(\log N)^{\gamma/c_1}}]\right)\\
&=\bigcup_{I'\in\mathcal{F}_1}\left(E_{I'}(I')+[-2{(2N_1+1)}^{d/2}e^{-c_5(\log N)^{\gamma/c_1}},2{(2N_1+1)}^{d/2}e^{-c_5(\log N)^{\gamma/c_1}}]\right){\color{red}.}
\end{align*}
Thus by \eqref{e622}, we obtain
\begin{align*}
\mathrm{mes}(E(I))&\leq \mathrm{mes}\left(\bigcup_{I'\in\mathcal{F}_1}E_{I'}(I')\right)+N_1^{C}e^{-c_5(\log N)^{\gamma/c_1}}\\
&\leq \mathrm{mes}\left(\bigcup_{I'\in\mathcal{F}_1 }E_{I'}(I')\right)+\frac{1}{N_1}.
\end{align*}
This proves \eqref{e624}.
\end{proof}

Now we can prove Theorem \ref{cor1}.
\begin{proof}[\bf Proof of Theorem \ref{cor1}]
Choose $N_s\sim e^{(\log N_{s-1})^{10}}$ ($s\geq 1$), where $N_0$ is given by Lemma \ref{l61}. Then applying Lemmas \ref{l61} and \ref{l62} yields a system $(I,E_I(\cdot))_{I\in \mathcal{F}_s}$ satisfying for $\theta\in I\in\mathcal{F}_s$,
\begin{align}\label{e631}
\mathrm{dist}(E_I(\theta),\widetilde{\Sigma})\leq e^{-c_5(\log N_s)^{\gamma/c_1}}.
\end{align}
Moreover, for any $s\geq 1$, one has
\begin{align*}
\mathrm{mes}\left(\bigcup_{I\in\mathcal{F}_s }E_{I}(I)\right)
&\geq\mathrm{mes}\left(\bigcup_{I\in\mathcal{F}_{s-1} }E_{I}(I)\right)-\frac{1}{N_{s}} \\
&\geq \mathrm{mes}(E_{I_0}(I_0))-\sum_{s\geq1}\frac{1}{N_s}\\
&\geq \frac{\mathrm{mes}(E_{I_0}(I_0))}{2},
\end{align*}
where $E_{I_0}, I_0$ are given by Lemma \ref{l61}.
Define
\begin{align*}
\Omega=\bigcap_{s\geq 0}\bigcup_{I\in\mathcal{F}_s }E_{I}(I).
\end{align*}
Since \eqref{e631}, we have
\begin{align*}
\Omega\subset\widetilde{\Sigma}
\end{align*}
and
\begin{align*}
\mathrm{mes}(\Omega)\geq \frac{\mathrm{mes}(E_{I_0}(I_0))}{2}.
\end{align*}
Thus it suffices to establish some lower bound on $\mathrm{mes}(E_{I_0}(I_0))$.

Recall that $E_{I_0}(\cdot)$ is continuous on $I_0$ and $|I_0|\geq N_0^{-C_1}$. We can write $E_{I_0}(I_0)=[E_0+\varepsilon,E_0-\varepsilon]$ for some $E_0\in E_{I_0}(I_0)$ and $\varepsilon\geq 0$. It needs to establish some concrete lower bound on $\varepsilon$. Choose $\underline{N}_0\leq M\ll N_0 $ and apply the LDT (i.e. Theorem \ref{ldt}) at scale $M$, where $M$ will be specified later. We have
\begin{align*}
\| G_{M}(E_0;\theta)\|&\leq e^{{{M}}^{\gamma/2}},\\
|G_{M}(E_0;\theta)(n,n')|&\leq e^{-{\frac{(1-5^{-\gamma})\rho}{2}}|n-n'|^{\gamma}}\  {\mathrm{for} \ |n-n'|\geq {{M}}/{10}}
\end{align*}
provided $\theta$ is outside a set $\Theta\subset [0,1]$ with $\mathrm{mes}(\Theta)\leq e^{-M^{c_1}}$. Paving $[-N_0,N_0]^d$ with $Q\in\mathcal{E}_M$, then we have by Lemma \ref{res1},
\begin{align}\label{e632}
\|G_{N_0}(E_0;\theta)\|\leq (10M)^de^{M^{\gamma/2}}\leq e^{2M^{\gamma/2}}
\end{align}
provided $\theta$ is outside a set $\Theta_1\subset [0,1]$ with $\mathrm{mes}(\Theta_1)\leq (10N_0)^de^{-M^{c_1}}$. Fix
\begin{align*}
M\sim (\log N_0)^{3/(2c_1)}.
\end{align*}
Then
\begin{align*}
(10N_0)^de^{-M^{c_1}}< \frac{N_0^{-C_1}}{2}
\end{align*}
 and thus  $([0,1]\setminus \Theta_1)\cap I_0\neq\emptyset$. We pick a $\theta_0\in([0,1]\setminus \Theta_1)\cap I_0$  and  a $\xi$ with $\|\xi\|=1$ so that
\begin{align*}
|(\widetilde H(\theta_0)-E_{I_0}(\theta_0))\xi\|\leq e^{-c_5(\log N_0)^{\gamma/c_1}}.
\end{align*}
Note that
\begin{align}
\nonumber\|(\widetilde H_{N_0}(\theta_0)-E_0)\xi\|&= \| (\widetilde H(\theta_0)-E_{I_0}(\theta_0))\xi-(E_0-E_{I_0}(\theta_0))\xi\|\\
\nonumber&=\|(\widetilde H(\theta_0)-E_{I_0}(\theta_0))\xi-R_{\Z^d\setminus [-N_0,N_0]^d}\widetilde H(\theta_0)\xi\\
\nonumber&\ \ -(E_0-E_{I_0}(\theta_0))\xi\|\\
\label{e633}&\leq 2e^{-c_5(\log N_0)^{\gamma/c_1}}+\varepsilon.
\end{align}
Recalling \eqref{e632}, we have
\begin{align}\label{e634}
\|G_{N_0}(E_0;\theta_0)\|\leq e^{2(\log N_0)^{3\gamma/(4c_1)}}.
\end{align}
Combining \eqref{e633} and \eqref{e634} yields
\begin{align*}
e^{-2(\log N_0)^{3\gamma/(4c_1)}}\leq 2e^{-c_5(\log N_0)^{\gamma/c_1}}+\varepsilon
\end{align*}
and 
\begin{align*}
\varepsilon \geq \frac{1}{2}e^{-2(\log N_0)^{3\gamma/(4c_1)}}.
\end{align*}

In conclusion, we have shown
\begin{align*}
\mathrm{mes}(\widetilde{\Sigma})\geq e^{-10(\log N_0)^{3\gamma/(4c_1)}}>0.
\end{align*}

This proves Theorem \ref{cor1}.
\end{proof}

\section*{Acknowledgements}
I would like to thank  Svetlana Jitomirskaya  for reading the earlier versions of the paper
and her constructive suggestions.  I am very grateful to  the anonymous  referee for carefully reading
the paper and providing many valuable comments that improved the exposition of
the paper. 
 %
This work was supported by NNSF  of China  grant 11901010.

\appendix
\section{}

 We write $G_{(\cdot)}=G_{(\cdot)}(E;\theta)$ for simplicity. Let $\Lambda_1,\Lambda_2\subset \Z^d$ with $\Lambda_1\cap\Lambda_2=\emptyset$. Let $\Lambda=\Lambda_1\cup \Lambda_2$.
 If $m\in \Lambda_1$ and $n\in \Lambda$, we have
 \begin{equation}\label{Greso}
    G_{\Lambda}(m,n)= G_{\Lambda_1}(m,n)\chi_{\Lambda_1}(n){-\lambda}\sum_{n^{\prime}\in \Lambda_1,n^{\prime\prime}\in \Lambda_2} G_{\Lambda_1}(m,n^{\prime})\mathcal{T}_v(n',n'')G_{\Lambda}(n^{\prime\prime},n).
 \end{equation}

We first prove a useful perturbation argument (see Lemma A.1 of \cite{SJDE} for a more general form with $\gamma=1$). 
\begin{lem}\label{pag}
Fix $\bar\rho>0$. Let $\Lambda\subset\mathbb{Z}^d$ satisfy $\Lambda\in\mathcal{E}_N$ and let $A,B$ be two linear operators  on $\C^\Lambda$.  We assume
 \begin{align*}
\|A^{-1}\|&\leq e^{N^{\gamma/2}},\\
 |A^{-1}(n,n')|&\leq e^{-\bar\rho|n-n'|^\gamma}\ \mathrm{for}\  |n-n'|\geq N/10.
 \end{align*}
Suppose that for all $n,n'\in\Lambda$, $$|(B-A)(n,n')|\leq e^{-3\bar\rho N^\gamma-\bar\rho|n-n'|^\gamma}.$$
Then
\begin{align*}
\|B^{-1}\|&\leq 2\|A^{-1}\|,\\
 |B^{-1}(n,n')|&\leq |A^{-1}(n,n')|+e^{-\bar\rho|n-n'|^\gamma}.
 \end{align*}
\end{lem}

\begin{proof}
Obviously,  $B=A(I+A^{-1}(B-A)).$
We write $P=A^{-1}(B-A)$. Then by the assumptions, $\|P\|\leq {1}/{2}$,
which together with the Neumann series argument implies
\begin{align*}
\|B^{-1}\|\leq \sum_{s\geq 0}2^{-s}\|A^{-1}\|=2 \|A^{-1}\|.
\end{align*}

Observing that  for any $m,n\in \Lambda$,
\begin{align*}
|A^{-1}(m,n)|\leq e^{N^{\gamma/2}+\bar\rho (N/10)^\gamma-\bar\rho|m-n|^\gamma},
\end{align*}
then for $m^0=m,m^s=n$ and $s\geq 1$, we have
$$P^{s}(m,n)=\sum_{m^1,\cdots,m^{s-1},n^1,\cdots,n^s\in \Lambda\ }\prod_{t=1}^{s}A^{-1}(m^{t-1},n^t)(B-A)(n^t,m^t).$$
Thus for $s\geq1$ and $N\gg1$, one has
\begin{align*}|P^{s}(m,n)|&\leq (CN)^{2sd}e^{s(N^{\gamma/2}-2\bar\rho N^\gamma){-\bar\rho|m-n|^\gamma}}\\
&\leq e^{-{3\bar\rho sN^\gamma}/{2}{-\bar\rho|m-n|^\gamma}}.\end{align*}
As a result, we obtain
\begin{align*}
|B^{-1}(n,n')|
&\leq|A^{-1}(n,n')|+\sum_{m\in\Lambda}\sum_{s\geq 1}|P^{s}(n,m)|\cdot|A^{-1}(m,n')|\\
&\leq|A^{-1}(n,n')|+\sum_{m\in\Lambda}\sum_{s\geq 1}e^{-{3\bar\rho sN^\gamma}/{2}{-\bar\rho|m-n|^\gamma}}\cdot|A^{-1}(m,n')|\\
&\leq|A^{-1}(n,n')|+\sum_{m\in\Lambda, |m-n'|\leq N/10}\sum_{s\geq 1}e^{-{3\bar\rho sN^\gamma}/{2}{-\bar\rho|m-n|^\gamma}+N^{\gamma/2}}\\
&\ \ +\sum_{m\in\Lambda, |m-n'|\geq N/10}\sum_{s\geq 1}e^{-{3\bar\rho sN^\gamma}/{2}{-\bar\rho|m-n|^\gamma}}e^{-\bar\rho|m-n'|^\gamma}\\
&\leq|A^{-1}(n,n')|+\sum_{m\in \Lambda,|m-n'|\leq N/10} e^{{-{\bar\rho N^\gamma}/{4}+N^{\gamma/2}-\bar\rho|n-n'|^\gamma}}\\
&\ \ +\sum_{m\in\Lambda,\ |m-n'|>N/10}e^{{-{\bar\rho N^\gamma}/{4}}-\bar\rho|n-n'|^\gamma}\\
&\leq|A^{-1}(n,n')|+e^{-\bar\rho|n-n'|^\gamma}.
\end{align*}
\end{proof}

\begin{lem}\label{res1}
 Let $ \bar\rho\in (\varepsilon,\rho]$, $M_1\leq N$ and ${\rm diam}(\Lambda)\leq 2N+1$. Suppose that for any $n\in \Lambda $, there exists some  $ W=W(n)\in \mathcal{E}_M$ with
$M_0\leq M\leq M_1$ such that
$n\in W\subset \Lambda$,  ${\rm dist} (n,\Lambda \backslash W)\geq {M}/{2}$ and
\begin{align}
\label{w1}\|G_{W}\|&\leq2 e^{{M}^{\gamma/2}},\\
\label{w2} |G_{W}(n,n')|&\leq  2e^{-\bar\rho|n-n'|^\gamma}\  {\mathrm{for} \ |n-n'|\geq {M}/{10}}.
\end{align} We assume  that $M_0\geq M_0(\varepsilon,\gamma,d)\gg1$.
Then
\begin{equation*}
  \|G_{\Lambda}\|\leq 4 (2M_1+1)^d e^{{M_1}^{\gamma/2}}.
\end{equation*}
\end{lem}
\begin{proof}
We fix $n,n'\in \Lambda$ and $W=W(n)$ as in the assumptions. Then $|W|\leq (2M+1)^d$.
By \eqref{w1} and \eqref{w2},
one has for all $k,k'\in W$,
\begin{align*}
   |G_{W}(k,k')|\leq 2e^{{M}^{\gamma/2}+{\bar{\rho}}(M/10)^\gamma}e^{-\bar\rho|k-k'|^\gamma}.
\end{align*}
Applying \eqref{Greso} with $\Lambda_1=W=W(n)$, one has
\begin{align}
 \nonumber |G_{\Lambda}(n,n')| &\leq  |G_{W}(n,n')|\chi_{W}(n') \\
\nonumber  &\ \ +2{\lambda}\sum_{n_1\in W\atop n_2\in \Lambda\backslash W} e^{{M}^{\gamma/2}+{\bar{\rho}}(M/10)^\gamma}e^{-\bar\rho|n-n_1|^\gamma-\rho|n_1-n_2|^\gamma} |G_{\Lambda}(n_2,n')|\\
  \nonumber &\leq  |G_{W}(n,n')|\chi_{W}(n')\\
  \nonumber &\ \ +2{\lambda}\sum_{n_1\in W\atop n_2\in \Lambda\backslash W} e^{{M}^{\gamma/2}+{\bar{\rho}}(M/10)^\gamma}e^{-\bar\rho|n-n_2|^\gamma} |G_{\Lambda}(n_2,n')|\\
   \nonumber&\leq|G_{W}(n,n')|\chi_{W}(n')\\
 \nonumber &\ \ +2{\lambda}(2M+1)^de^{{M}^{\gamma/2}+{\bar{\rho}}(M/10)^\gamma}\sum_{n_2\in \Lambda\atop |n_2-n|\geq {M}/{2}}   e^{-\bar\rho|n-n_2|^\gamma}|G_{\Lambda}(n_2,n')|\\
 \label{BGSle1}&\leq|G_{W}(n,n')|\chi_{W}(n')+2{\lambda}(2M+1)^de^{{M}^{\gamma/2}-\varepsilon(M/10)^\gamma}\sup_{n_2\in\Lambda}|G_{\Lambda}(n_2,n')|,
\end{align}
where the third inequality holds since ${\rm dist} (n,\Lambda\backslash W)\geq {M}/{2}$.
Summing over $n'\in \Lambda$ in \eqref{BGSle1} and using $M_0\geq M_0(\varepsilon,\gamma,d)\gg1$ yield ({since $0<\lambda<1$})
\begin{align*}
 \sup_{n\in \Lambda}\sum_{n'\in \Lambda} |G_{\Lambda}(n,n')|\leq2(2M_1+1)^de^{{M_1}^{\gamma/2}}+\frac{1}{2}\sup_{n_2\in \Lambda}\sum_{n'\in \Lambda}|G_{\Lambda}(n_2,n')|.
\end{align*}

This  lemma then follows from the Schur's test and the self-adjointness of $G_\Lambda$.

\end{proof}

\begin{lem}\label{res2}
Let $\Lambda_1\subset\Lambda\subset\Z^d$ satisfy ${\rm diam}(\Lambda)\leq 2N+1$ and $ {\rm diam}(\Lambda_1)\leq N^{\frac{\gamma}{3d}}$. Let $ M_0\geq (\log N)^{2/\gamma}$ and $\bar\rho\in [(1-5^{-\gamma})/10,(1-5^{-\gamma})\rho]$.
Suppose that for any $n\in \Lambda \backslash\Lambda_1$, there exists some  $ W=W(n)\in \mathcal{E}_M$ with
$M_0\leq M\leq N^{\gamma/3}$ such that
$n\in W\subset \Lambda\backslash\Lambda_1$, ${\rm dist} (n,\Lambda\backslash \Lambda_1\backslash W)\geq {M}/{2}$  and
\begin{align*}
\|G_{W}\|&\leq e^{{M}^{\gamma/2}},\\
 |G_{W}(n,n')|&\leq  e^{- \bar{\rho}|n-n'|^\gamma}\  {\mathrm{for} \ |n-n'|\geq {M}/{10}}.
\end{align*}
Suppose  that
\begin{equation}\label{gn}
\|G_{\Lambda}\|\leq e^{{N}^{\gamma/2}}.
\end{equation}
Then
\begin{align*}
 |G_{\Lambda}(n,n')|\leq e^{-(\bar{\rho}-\frac{C}{M_0^{\gamma/2}})|n-n'|^\gamma}\ \mathrm{for}\  |n-n'|\geq{N}/{10},
\end{align*}
where $C=C(d,\rho,\gamma)>0$.
\end{lem}
\begin{proof}
We first assume $n\in\Lambda\setminus\Lambda_1,n'\in\Lambda_1$ and $|n-n'|\geq N^{\gamma/2}$. We let $W=W(n)\subset\Lambda\setminus\Lambda_1$ satisfy the assumptions as above. Note that for $|n-n_2|\geq M/2$ and $0<\rho<(1-5^{-\gamma})\rho$, one has
\begin{align}\label{rhobar}
e^{-\rho|n-n_2|^\gamma+\rho(M/10)^\gamma}\leq e^{-\bar\rho|n-n_2|^\gamma}.
\end{align}
Recall that  $0<\lambda<1$ and $|n-n'|\geq N^{\gamma/2}>10N^{\gamma/3}>\mathrm{diam}(W)$.  Applying \eqref{Greso} with $\Lambda_1=W=W(n)$ yields
\begin{align}
\nonumber|G_{\Lambda}(n,n')|
&\leq \nonumber  \sum_{n_1\in W,|n_1-n|\leq \frac{M}{10}\atop n_2\in \Lambda\backslash W} e^{{M}^{\gamma/2}}e^{-\rho|n_1-n_2|^\gamma}|G_{\Lambda}(n_2,n')|\\
\nonumber&\ \ + \sum_{n_1\in W,|n_1-n|\geq \frac{M}{10}\atop n_2\in \Lambda\backslash W} e^{-\bar{\rho} |n-n_1|^\gamma}e^{-\rho|n_1-n_2|^\gamma} |G_{\Lambda}(n_2,n')|\\
&\leq \nonumber  \sum_{n_1\in W,|n_1-n|\leq \frac{M}{10}\atop n_2\in \Lambda\backslash W} e^{{M}^{\gamma/2}}e^{-\rho|n-n_2|^\gamma+\rho(M/10)^\gamma}|G_{\Lambda}(n_2,n')|\\
\nonumber&\ \ + \sum_{n_1\in W,|n_1-n|\geq \frac{M}{10}\atop n_2\in \Lambda\backslash W} e^{-\bar\rho|n-n_2|^\gamma} |G_{\Lambda}(n_2,n')|\\
&\leq \nonumber \sum_{n_1\in W,|n_1-n|\leq \frac{M}{10}\atop n_2\in \Lambda\backslash W} e^{{M}^{\gamma/2}}e^{-\bar\rho|n-n_2|^\gamma}|G_{\Lambda}(n_2,n')|\ (\mbox{by \eqref{rhobar}})\\
\nonumber&\ \ + \sum_{n_1\in W,|n_1-n|\geq \frac{M}{10}\atop n_2\in \Lambda\backslash W} e^{-\bar\rho|n-n_2|^\gamma} |G_{\Lambda}(n_2,n')|\\
\label{rsi8}&\leq  2(2N+1)^{2d}\sup_{n_2\in \Lambda\backslash W}  e^{-(\bar\rho-\frac{C}{M_0^{\gamma/2}})|n-n_2|^\gamma}|G_{\Lambda}(n_2,n')|,
\end{align}
where the last inequality holds because of $|n-n_2|\geq {M}/{2}$ and $M\geq M_0$.
Iterating \eqref{rsi8} until $|n_2-n'|\leq N^{{\gamma}/2}$ (but stopping at most $\frac{C|n-n'|^\gamma}{M_0^\gamma}$ steps), we have since $|n-n'|\geq N^{\gamma/2}$ and $M_0\geq(\log N)^{2/\gamma}$,
\begin{align*}
  |G_{\Lambda}(n,n')| &\leq   (10N)^{ \frac{C|n-n'|^\gamma}{M_0^\gamma}} e^{-(\bar\rho-\frac{C}{M_0^{\gamma/2}})(|n-n'|^\gamma-N^{{\gamma^2}/2}) }e^{{N}^{\gamma/2}}\\
   &\leq e^{-(\bar{\rho}-\frac{C}{M_0^{\gamma/2}}-\frac{C\log N}{M_0^\gamma})|n-n'|^\gamma+2N^{\gamma/2}}\ \mbox{(since $0<\rho<1$)}\\
   &\leq  e^{-(\bar{\rho}-\frac{C}{M_0^{\gamma/2}})|n-n'|^\gamma+2N^{\gamma/2}}.
\end{align*}
Recalling  \eqref{gn}  again, we obtain for all $n\in\Lambda\setminus\Lambda_1,n'\in\Lambda_1$,
\begin{align*}
|G_{\Lambda}(n,n')|\leq  e^{-(\bar{\rho}-\frac{C}{M_0^{\gamma/2}})|n-n'|^\gamma+3N^{\gamma/2}}.
\end{align*}
Then by the self-adjointness of $G_{\Lambda}$, one has for $n\in\Lambda_1,n'\in\Lambda\setminus\Lambda_1$,
\begin{align}\label{bad}
|G_{\Lambda}(n,n')|\leq  e^{-(\bar{\rho}-\frac{C}{M_0^{\gamma/2}})|n-n'|^\gamma+3N^{\gamma/2}}.
\end{align}

We now assume $n,n'\in \Lambda$ satisfy $|n-n'|\geq N^{\gamma/2}$. By $\mathrm{diam}(\Lambda_1)\leq N^{\frac{\gamma}{3d}}$, at least one of $n,n'$ must be in $\Lambda\setminus\Lambda_1$. From the above discussions, it remains  assuming $n,n'\in\Lambda\setminus\Lambda_1$. Similar to the proof of \eqref{rsi8}, we have
\begin{align}\label{rsi8new}
|G_{\Lambda}(n,n')|&\leq  2(2N+1)^{2d}\sup_{n_2\in \Lambda\backslash W}  e^{-(\bar\rho-\frac{C}{M_0^{\gamma/2}})|n-n_2|^\gamma}|G_{\Lambda}(n_2,n')|,
\end{align}
where $|n-n_2|\geq M/2$. Hence iterating \eqref{rsi8new} until $n_2\in\Lambda_1$ (but stopping at most $\frac{C|n-n'|^\gamma}{M_0^\gamma}$ steps), we have for $|n-n'|\geq N^{\gamma/2}$ (and some $n_2\in\Lambda_1$),
\begin{align*}
  |G_{\Lambda}(n,n')| &\leq   (10N)^{ \frac{C|n-n'|^\gamma}{M_0^\gamma}} e^{-(\bar\rho-\frac{C}{M_0^{\gamma/2}})|n-n_2|^\gamma}|G_{\Lambda}(n_2,n')|\\
   &\leq  (10N)^{ \frac{C|n-n'|^\gamma}{M_0^\gamma}} e^{-(\bar\rho-\frac{C}{M_0^{\gamma/2}})|n-n_2|^\gamma}e^{-(\bar{\rho}-\frac{C}{M_0^{\gamma/2}})|n_2-n'|^\gamma+3N^{\gamma/2}}\ (\mbox{by \eqref{bad}})\\
   &\leq  e^{-(\bar{\rho}-\frac{C}{M_0^{\gamma/2}})|n-n'|^\gamma+3N^{\gamma/2}}.
\end{align*}

Finally, since $|n-n'|\geq N/10$,  we have $\frac{N^{\gamma/2}}{|n-n'|^\gamma}\ll M_0^{-\gamma/2}$.

This finishes the proof.
\end{proof}


 \end{document}